\documentclass[a4paper, 11pt]{article}

\usepackage{etex}
\usepackage{amssymb, amsfonts, amsmath, amsthm, amscd}
\usepackage{mathrsfs, enumerate, caption, tabularx}
\usepackage{booktabs, vmargin, multirow, rotating, subfigure, ifthen}

\newboolean{appendice}
\setboolean{appendice}{false}

\usepackage[usenames,dvipsnames]{color}
\usepackage[english]{babel}

\usepackage{tikz}
\usetikzlibrary{%
  shapes,%
  calc,%
  decorations.markings}

\usepackage[matrix,arrow,curve]{xy}

\usepackage[latin1]{inputenc}

\setpapersize{A4}
\setmarginsrb{35mm}{25mm}{35mm}{25mm}%
             {0mm}{25mm}{0mm}{25mm}

\newtheorem{lemma}{Lemma}[section]

\newtheorem{prop}[lemma]{Proposition}
\newtheorem{thmA}{Theorem}
\newcounter{\thethmA}{\Alph{\thethmA}}
\newtheorem{cor}[lemma]{Corollary}

\theoremstyle{definition}
\newtheorem{defi}{Definition}
\newtheorem{assu}{Assumption}

\theoremstyle{remark}
\newtheorem{oss}{Remark}

\newtheorem{es}{Example}

\newcommand{\call}[1]{\ensuremath {\mathcal{#1}}}
\newcommand{\scr}[1]{\ensuremath {\mathscr{#1}}}

\newcommand{\sx} {\ensuremath {\left(}}
\newcommand{\dx} {\ensuremath {\right)}}
\newcommand{\mat}[1]{\ensuremath {\mathbb{#1}}}

\newcommand{\Hom} {\ensuremath \textrm{Hom}}

\newcommand{\Ker} {\ensuremath \textrm{Ker}\,}

\newcommand{\gen}[1]{\ensuremath \langle #1 \rangle}

\newcommand{\agisce} {\ensuremath \curvearrowright}

\renewcommand{\tilde}[1]{\ensuremath {{#1}^\upharpoonright}}

\renewcommand{\epsilon}{\ensuremath \varepsilon}

\let\oldmarginpar\marginpar
\renewcommand\marginpar[1]{\-\oldmarginpar[\raggedleft\footnotesize #1]%
  {\raggedright\footnotesize #1}}

\title{A Salvetti complex for Toric Arrangements and its fundamental group}

\author{Giacomo d'Antonio and Emanuele Delucchi}

\begin{document}
\maketitle
 
\begin{abstract}
We describe a combinatorial model for the complement of a complexified
toric arrangement by using nerves of
acyclic categories.
This generalizes recent work of
Moci and Settepanella \cite{mocisettepanella} on {\em thick} toric arrangements.

Moreover, we study its fundamental group and compute a presentation thereof.
\end{abstract}

\section*{Introduction}

A \emph{toric arrangement} is, roughly speaking, a family of
subtori of a complex torus $\sx\mat C^*\dx^n$. The study of the
topology and the combinatorics of such objects is a fairly new, yet
thriving topic. As the very first attempt in this direction we can
cite the work of Lehrer 
\cite{lehrer95toral}, where the representation theory
on the cohomology of the configuration space $F(\mat C^*, n)$ of $n$ points in the
pointed complex plane is studied.
This configuration space is indeed the complement of a toric arrangement. Its topology is already well known, since 
$F(\mat C^*, n) \simeq F(\mat C, n+1)$.

The foundation of the topic can be traced to the paper 
\cite{de2005geometry} by De Concini and Procesi. There the main objects are defined, the cohomology of the complement 
of a toric arrangement is studied (mainly from the point of view
of algebraic geometry) and some applications of the theory are outlined.
In particular, these authors treat the topic with the explicit goal
of generalizing the theory of hyperplane arrangements, and they put
all this in a wider context that encompasses applications in
topics such as the study of integer points of Zonotopes and box
splines. An extensive account of the work of De Concini and Procesi on this new
subject can be found in their forthcoming book \cite{deconcini2010topics}.

Ehrenborg, Readdy and Slone \cite{ehrenborg2009affine} take another
point of view, studying 
toric arrangements on the ``compact torus'' $(S^1)^n$ and considering the
problem of enumerating faces of the induced decomposition of the
compact torus.

The next step is the work of Moci, in particular
his papers \cite{mocicombinatorics}, \cite{mocitutte2009} and
\cite{mociwonderful2009}, developing the theory with a
special focus on combinatorics. In particular, Moci introduces a
two-variable polynomial that encodes enumerative invariants of many
of the different objects populating the landscape outlined by De
Concini and Procesi in \cite{deconcini2010topics}. The same author, in joint
work with Settepanella \cite{mocisettepanella}, studied the homotopy
type of the complement of a special class of toric arrangements ({\em
  thick} arrangements, see Section \ref{sec:toric} below). In this work we will use a
similar but more general approach, so that our results hold for a
wider class of toric arrangements, which we call {\em complexified}
because of structural affinity with the case of hyperplane arrangements.

Indeed, a rich and lively theory exists for arrangements of
hyperplanes in affine complex space. An affine hyperplane is the
(translate of the) kernel of a linear form. An affine arrangement is
called complexified if the defining linear forms are real linear
forms. Equivalently, a complexified arrangement induces an arrangement
of real (affine) hyperplanes that determines it completely. It is from
this equivalent formulation that we take inspiration for our
definition of {\em complexified toric arrangements}: these are the
arrangements that induce an arrangement in the compact torus and are
determined by it. Every `thick' arrangement in the sense of
\cite{mocisettepanella} is complexified, and there are nonthick
complexified arrangements.\\

It is our explicit goal to try to present the theory and the results in a way
that at once underlines the structural similarities with the theory of
hyperplane arrangements and shows where (and why) the peculiarities of
the toric theory are. 

We will try to do so by using a
combinatorial tool that aptly generalizes the idea of a poset and its order complex:
{\em acyclic categories} and their nerves.

Our first main result shows that the combinatorial
structure of a complexified toric arrangement can be used to construct an
acyclic category whose nerve is homotopy equivalent to the complement
of the arrangement. It is this acyclic category that we suggest to
call {\em Salvetti category}.
Accordingly, we suggest to call the complex
obtained as the nerve of the Salvetti category the {\em Salvetti
  complex} of the toric arrangement. Our result specializes to the
construction of \cite{mocisettepanella} for the case of thick arrangements.

The second main result is the computation of a (finite) presentation
for the fundamental group of the arrangement's complement, appearing
here for the first time, to the best of our knowledge. \\

Our paper begins with a review of the relevant background facts about
hyperplane arrangements and acyclic categories: this will be the content of Section \ref{sec:background}.
Then, in Section \ref{sec:toric} we give a brief account of the theory
of toric arrangements, with the special goal to set some notations,
terminology and basic facts that will be relevant for the sequel.
With Section \ref{sec:salvetti} we will enter the core ouf our
work, defining our combinatorial model 
(Definition \ref{df:salvetticategory}) and proving our first main result (Theorem
\ref{teo:teoremone}): the nerve of the Salvetti category
models the homotopy type of any complexified toric arrangement.
The computation of our presentation for the fundamental group
will be carried out in Section \ref{sec:fg}, and the presentation
itself will be given as our second main result, Theorem \ref{teo:mainfg}.

\subsubsection*{Acknowledgements} The second author was introduced to
this subject by Luca Moci, whom he thanks for the many interesting and
insightful conversations.

The two authors started their collaboration during the special period on arrangements and
configuration spaces at the Centro De Giorgi of the Scuola Normale
Superiore. This work was mainly carried out at the department of
mathematics of the University of Bremen.

Both authors gladly acknowledge conversations with prof. Mario
Salvetti, thanking him for pointing out a flaw in Section 4 of an early version of
this paper.

\section{Background}
\label{sec:background}
\subsection{Arrangements of hyperplanes}
\label{sec:hyperplanes}
Before turning our attention to toric arrangements, let us briefly review
some basics about hyperplane arrangements.

Let families of linear
forms $l_1,\ldots l_n \in \Hom(\mathbb C ^d, \mathbb C)$ and
scalars $z_1,\ldots z_n$ be given. For every $i=1,\ldots, n$ we have then an
affine hyperplane
$$H_i:=\{z\in \mathbb C^d : l_i(z)=z_i\}.$$ 
The {\em (affine) hyperplane arrangement} in $\mathbb C^d$ defined by
the given linear forms and scalars is the set $$\scr A
= \{H_1,\ldots , H_n\}. $$  

The arrangement is called {\em complexified} if its defining forms
are real, i.e., $l_i\in\Hom(\mathbb R^d , \mathbb R)$ for all $i$.

There are several descriptions of the homotopy type of the complement of
a set of hyperplanes in complex space.
In this paper we will take inspiration by the work of Salvetti
\cite{salvetti1987tcr}, where a regular polytopal complex 
which embeds in the complement of a 
{complexified real arrangement} as a deformation retract is
constructed: the \emph{Salvetti complex}.

\begin{defi}\label{def:salvettiposetaffine}
  Let $\scr A$ be a complexified real arrangement in
  $\mat C^n$. We write $\call D = \call D(\scr A)$
  for the cellular decomposition induced by $\scr A$ on $\mat R^n$ and
  $\call F = \call F(\scr A)$ for its face poset (ordered by
  inclusion\footnote{The reader should be aware that this is in
    contrast to some of the existing literature.}). The maximal
  elements of $\call F$ are called {\em chambers}.
\end{defi}

Given a face $F \in \call F$, we can consider the affine
subspace $|F|$ it generates, say $|F|=y+L$ for a linear subspace $L$. The projection map
$\pi_F: \mat R^n \to \mat R^n/L$ maps chambers of
$\scr A$ on  chambers of 
the arrangement
\begin{equation}\label{eq:quotienthyperplanes}
  \scr A_F = \{\pi_F(H):\, F \subseteq H\}.
\end{equation}

 We define the
  \emph{Salvetti poset} $\mbox{Sal}(\scr A)$ on the element set
  \begin{displaymath}
    \{[F, C]:\, F,C \in \call D \mbox{ and } F \leq C \mbox{ in }
    \call F\}
  \end{displaymath}
  by the order relation
  \begin{equation}
    \label{eq:salvettiaffine}
    [F_1,C_1] \leq [F_2,C_2] \iff
    F_2 \leq F_1 \mbox{ in } \call F \mbox{ and }
    \pi_{F_1}(C_2) = \pi_{F_1}(C_1).
  \end{equation}

\begin{defi}\label{def:salvettiaffine}
  Let $\scr A$ be a complexified real arrangement in $\mat C^n$;
  the \emph{Salvetti complex} of $\scr A$ is the simplicial complex
  $\call S = \call S(\scr A) := \Delta(\mbox{Sal}(\scr A))$.
\end{defi}

\begin{prop}[Salvetti \cite{salvetti1987tcr}]
The complex $\call S (\scr A)$ is a deformation retract of the
arrangement's complement, i.e., of the space $\mathbb C^d \setminus \bigcup_{i=1}^n H_i$
\end{prop}

The simplicial complex $\call S$ is the barycentric subdivision of a 
regular polytopal complex that we want now to describe.

Consider the graph $\call G(\scr A)$ with  the set of
chambers of $\scr A$ as vertex set and edge set given by 
\begin{displaymath}
  E = \{e_{[F,C]} = (C,D):\, F \in \call D,\, \mbox{codim}\,(F) = 1,\,
  F \leq C,\, \mbox{op}(C,F) = D\}
\end{displaymath}
where $\mbox{op}(C,F)$ is the \emph{opposite chamber} of $C$ with
respect to $F$. We can assign a direction to an edge $e_{[F,C]}$ by
thinking it oriented from $C$ to $\mbox{op}(C,F)$.
We say that every edge $e_{[F,C]}$ of $\call G(\scr A)$
`crosses' the hyperplane which supports $F$. A hyperplane $H$
separates two chambers $C$ and $D$ if a straight line segment from any
point in the interior of $C$ to any point in the interior of $D$
intersects $H$.

A path in $\call G(\scr A)$ from a vertex (chamber) $C$ to a vertex
(chamber) $D$ is {\em positive minimal} if it is directed and if it
never crosses any hyperplane more than once.

\begin{defi}\label{def:unsubdivided}
  The \emph{unsubdivided Salvetti complex} is the polytopal complex
  \begin{enumerate}[(i)]
  \item whose $1$-skeleton is the realization of the graph $\call G(\scr A)$;
  \item whose $k$-cells corresponds to the pairs $[F,C]$ with $F \in
    \call F (\scr A)$ a
    face of codimension $k$ and $C$ a chamber with $F \leq C$;
  \item where the $1$-skeleton of a $k$-cell $e_{[F,C]}$ is attached along
    the \emph{minimal positive directed paths} in $\call G(\scr A)$ from
    $C$ to $\call{OP}(C,F)$.
  \end{enumerate}
\end{defi}
The reader can now easily convince her- or himself that condition
\eqref{eq:salvettiaffine} states exactly when a cell $e_{[F_1,C_1]}$
lies in the boundary of the cell $e_{[F_2,C_2]}$ in the unsubdivided Salvetti 
complex. In other words, the poset $\mbox{Sal}(\scr A)$ is the face
poset of the unsubdivided Salvetti complex (and hence $\call S$ is its
barycentric subdivision).

We close this section by noting that the coarser structure of the
unsubdivided complex has been used already in the seminal paper by
Salvetti \cite{salvetti1987tcr} to compute the fundamental group of
the complement of a complexified hyperplane arrangement. We will
return to this topic and review the techniques introduced by Salvetti
when we will compute our presentation for the fundamental group of
complexified toric arrangements.

\subsection{Acyclic categories}

Let us now introduce the idea of acyclic categories. We can think of acyclic categories as posets in which more than one relation
between two elements is allowed. Our main general reference for this topic is
Kozlov's book \cite{kozlov2007combinatorial} and, for specifics about
 actions of infinite groups, Babson and Kozlov's paper \cite{Babson2005439}.

\begin{defi}
  An \emph{acyclic category} is a small category $C$, 
  such that:
  \begin{enumerate}[(i)]
  \item the only morphisms that have inverses are the identities;
  \item the only morphism from an object to itself is the identity.
  \end{enumerate}
\end{defi}
We will write $\call O(C)$ for the objects of $C$ and $\call M(C)$ for
its morphisms.

Acyclic categories occur sometimes in the literature as ``loop-free categories'' 
or ``scwol''s (small category without loops, cfr. \cite{bridson1999metric}).

\subsubsection{The nerve}
To an acyclic category we can associate its \emph{nerve}. This is the
generalization of the order complex of a poset. Meaning that, if the
category is indeed a poset (that is, between two arbitrary objects there
is at most a morphism), then its nerve is indeed its order complex.
In general, however, the nerve of an acyclic category will not be a simplicial
complex. Instead it will be a \emph{regular trisp}. Trisps --also called
\emph{$\Delta$-complexes} in \cite{hatcher2002at}-- are a generalization
of simplicial complexes.

To define trisps we start with the notion of a \emph{polytopal
  complex}. This is, roughly speaking, a complex obtained gluing polytopal cells.
We will follow Kozlov's book (\cite[Definition 2.39]{kozlov2007combinatorial}),
except that we don't require polytopal complexes to be \emph{regular}.
More precisely:
\begin{defi}
  A \emph{polytopal complex} is a topological space $X$ obtained
  with the following construction:
  \begin{enumerate}[(i)]
  \item Start with the $0$-skeleton $X_0$, a discrete set of points.
  \item At the $k$-th step we attach all the $k$-dimensional faces.
    These are convex polytopes $P \subseteq \mat R^k$, attached
    along the maps $f:\partial P \to X_{k-1}$. The \emph{attaching maps}
    are required to be \emph{cellular}.
    Furthermore, the interior
    of each face of $P$ has to be attached homeomorphically to
    the interior of a face in $X_{k-1}$. The $k$-skeleton is defined as
    \begin{displaymath}
      X_{k} = \sx\bigsqcup P \sqcup X_{k-1}\dx/_{x \sim f(x)}
    \end{displaymath}
  \item We define $X = \cup_{k \in \mat N} X_k$.
  \end{enumerate}
\end{defi}

A {\em trisp} can be described then as a polytopal complex in which
every cell is a simplex. For more details about trisps and for the precise definition we refer to
\cite{kozlov2007combinatorial}.

Having introduced trisps, we can now define the nerve of an acyclic category.
\begin{defi}
  Let $C$ be an acyclic category; the \emph{nerve} $\Delta(C)$ is
  the trisp
  \begin{enumerate}[(i)]
  \item whose $k$-dimensional simplexes are $k$-length chains of
    composable morphisms 
    \begin{displaymath}
      \sigma = a_0 \stackrel{m_1}{\to} a_1
      \stackrel{m_2}{\to} a_2 \stackrel{m_3}{\to} \cdots \stackrel{m_k}{\to} a_k,
    \end{displaymath}
  \item where the boundary simplexes of a simplex $\sigma$ as above are defined as:
    \begin{align*}
      & \partial_0 \sigma = a_1 \stackrel{m_2}{\to} a_2 
      \stackrel{m_3}{\to} \cdots \stackrel{m_k}{\to} a_k\\
      & \partial_j \sigma = a_0 \stackrel{m_1}{\to} \cdots
      \stackrel{m_{j-1}}{\to} a_{j-1} \stackrel{m_{j+1}\circ m_j}{\to} 
      a_{j+1} \stackrel{m_{j+2}}{\to}\cdots \stackrel{m_k}{\to} a_k\\
      & \partial_k \sigma = a_0 \stackrel{m_1}{\to} a_1 \stackrel{m_2}{\to} a_2 
      \stackrel{m_3}{\to} \cdots \stackrel{m_{k-1}}{\to} a_{k-1}
    \end{align*}
  \end{enumerate}
\end{defi}

\subsubsection{Face category}
\label{sec:category}

Acyclic categories can be used to describe the topology
of a polytopal complex. For this section we refer to
\cite[III $\call C$.1]{bridson1999metric}.

\begin{defi}\label{def:facecat}
  Let $X$ be a polytopal complex; its \emph{face category}
  is the acyclic category $\call F(X)$
  \begin{enumerate}[(i)]
  \item whose set of objects $\call O(\call F(X))$ corresponds
    to the set of cells of $X$,
  \item where for every cell $P$ of $X$ and for every face $F$ of the polytope $P$
    there is a morphism $m_{P,F}:Q \to P \in \call M (\call F (X))$, where $Q$ is the face of $X$ upon
    which $F$ is glued,
  \item where if  $P_3 \stackrel{m_{P_2,F_2}}{\to} P_2 \stackrel{m_{P_1,F_1}}{\to} P_1$
    is a composable chain of morphisms in $\call F(X)$, then
    \begin{displaymath}
      m_{P_1,F_1} \circ m_{P_2,F_2} = m_{P_1,F'}
    \end{displaymath}
    (here $F'$ is the face of $F_1$ which is glued upon $F_2 \subseteq P_2$, 
     and hence upon $P_3$).
  \end{enumerate}
\end{defi}

\begin{oss}
  We notice that in point (iii) of definition \ref{def:facecat} the face
  $F'$ is uniquely determined, since the (restriction of the)
  gluing map $F_2 \to P_2$ is a cellular homeomophism.
\end{oss}

\begin{defi}
  The \emph{barycentric subdivision} of a polytopal complex $X$, is the regular
  trisp $\call B(X) = \Delta(\call F(X))$: the nerve of the face category.
\end{defi}

The face category describes the topology of a polytopal complex in the following
sense:
\begin{prop}
  Let $X$ be a polytopal complex, then the geometric realization of
  $\call B(X)$ is homeomorphic to $X$.
\end{prop}

These concepts have been already used in metric geometry and especially
in geometric group theory. There acyclic categories are called \emph{scwol}s,
the nerve of a category is called the \emph{geometric realization}
and the face category of a polytopal complex is called
the \emph{barycentric subdivision}. More details can be found in
\cite[III$\call C$]{bridson1999metric}.

%%%%%%%%%%%%%%
\section{Toric arrangements}\label{sec:toric}
We will now introduce toric arrangements together with some
construction that will be needed in the following.

The $n$-dimensional {\em complex} torus is the space $(\mat C^*)^n$;
the $n$-dimensional {\em compact} torus is $(S^1)^n$.
A \emph{character} of a complex torus $T$ is an \emph{affine homomorphism}
$\chi: T \to \mat C^*$, i.e., a {\em Laurent polynomial} in
$\mat C[x_1^{\pm 1}, \dots x_n^{\pm 1}]$ that is also a group homorphism
with respect to the complex multiplication. One can easily see
that, then, $\chi$ is a Laurent {\em monomial} and for $x\in T$ we have
\begin{displaymath}
  \chi(x)=x_1^{\alpha_1} x_2^{\alpha_2} \cdots
  x_n^{\alpha_n}\quad\textrm{with}\quad
  \alpha = (\alpha_1, \dots, \alpha_n) \in \mat Z^n.
\end{displaymath}
The correspondence between  a character $\chi\in\Lambda$ and the
associated integer vector $\alpha_\chi$ makes the set of characters
into a lattice $\Lambda \cong \mat Z^n$ with the operation defined by
pointwis multiplication of characters.\\

The above, ``concrete'' definitions suffice for many purposes. 
It is however convenient for us and common in the literature to give a 
more abstract definition, starting with any (finitely generated) lattice $\Lambda$,
which will be our character lattice. We then define the corresponding
torus to be $$T_\Lambda := \Hom_{\mat Z}(\Lambda, \mat C^*).$$
Choosing a basis for $\Lambda$ gives an isomorphism $T_\Lambda \cong 
(\mat C^*)^{\mbox{rk}\,\Lambda}$ whose components are the
evaluation maps on the elements of the basis. Analogously, the
{\em compact} torus on the lattice $\Lambda$ is defined as 
$\Hom_{\mat Z}(\Lambda, S^1)$.

\begin{defi}\label{def:arrangement}
  A \emph{complexified toric arrangement} is a finite collection
  \begin{displaymath}
    \scr A = \{(\chi, a): \chi \in \Lambda,\, a \in S^1\},
  \end{displaymath}
where $\Lambda$ is a finitely generated lattice.
  We may think of $\scr A$ as the arrangement of the 
  hypersurfaces $H_{\chi,a} = \{x \in T_\Lambda:\, \chi(x) = a\}$, where
  $(\chi,a)$ runs over $\scr A$.

The {\em complement} of $\scr A$ is then
$$
M(\scr A) := (\mat C^*)^n \backslash \bigcup_{(\chi, a) \in \scr A}H_{\chi, a}.
$$
\end{defi}

\begin{oss}\label{oss:definizioni}
  Toric arrangements were first defined in \cite{de2005geometry} as sets of pairs 
  $(\chi, a)$ with $a \in \mat C^*$.
  Restricting the constants to $S^1$ allows for the same $\scr A$
  to define an arrangement of subtori on the compact torus
  $(S^1)^n$ (since a Laurent monomial maps $(S^1)^n$ on $S^1$). 
  The analogy with the case of complexified hyperplane arrangements
  motivates our terminology.  
  \end{oss}
\begin{defi} \label{def:D} Let $\scr A$ be a complexified toric arrangement.  With $\call D = \call D(\scr A)$
  we will denote the induced  cell-decomposition of the compact torus $(S^1)^n$.
\end{defi}
\begin{oss}
  On the other hand, \cite{mocitutte2009} and \cite{mocisettepanella} define
  a toric arrangement as an arrangement of kernels of characters (thus
  requiring $a = 1$). This cuts out a whole class of arrangements
  (e.g. $\scr A = \{t = -1, s = -1\}$ in $(\mat C^*)^2$).
  Moreover one can have hypersurfaces with many connected
  components, which are not in general kernels of characters (e.g. $t^2 = 1$).
\end{oss}

\begin{defi}
  A toric arrangement $\scr A$ on a $k$-dimensional
  torus $T_\Lambda$ is called \emph{essential} if
  \begin{displaymath}
    \mbox{rk}\,\scr A := 
    \mbox{rk}\,\left\langle\chi \in \Lambda:\,(\chi, a) \in \scr A
      \mbox{ for some } a \in S^1\right\rangle = k.
  \end{displaymath}
  This can be stated equivalently by saying that the layers of maximal codimension are points.
\end{defi}

\begin{oss}
  Consider a (non essential) arrangement $\scr A = \{(\chi_1,a_1), \dots,
  (\chi_n, a_n)\}$ with $\mbox{rk}\,\scr A = l < k$.
  Then there exists an essential arrangement $\scr A'$ 
  (the \emph{essentialisation} of $\scr A$) such that
  \begin{displaymath}
    M(\scr A) = M(\scr A') \times \sx\mat C^*\dx^{k - l}.
  \end{displaymath}
  With the notation of Definition \ref{defi:restriction},
  $\scr A' = \scr A_\Gamma$ where
  \begin{displaymath}
    \Gamma = \{\chi \in \Lambda:\quad \exists k \in \mat Z:\, \chi^k \in
    \gen{\chi_1, \dots, \chi_n}\}.
  \end{displaymath}

  In other words, it is not restrictive to consider essential arrangements. 
\end{oss}

\begin{assu}
 Unless otherwise stated, we will always assume our arrangement to
  be \emph{complexified} and {\em essential}.
\end{assu}

\begin{oss} As is the case in the theory of hyperplane arrangements, one of the goals of
the study of toric arrangements is to relate topological properties
of the complement 
$M(\scr A)$ 
to the \emph{combinatorics} of the arrangement $\scr A$. In the hyperplane case,
the combinatorics is expressed by the \emph{poset of intersections} 
$\call L(\scr A)$ of elements of $\scr A$. 
In the case of toric arrangements, the results of
\cite{de2005geometry} suggest that the right combinatorial
invariant may be the \emph{poset of layers} $\call C(\scr A)$, where a
{\em layer} is a connected component of an intersection of hypersurfaces
$H_{\chi, a}$, and the partial order is given by inclusion.

In the case of hyperplane arrangements, $\call L(\scr A)$ does not
suffice to determine the homotopy type of the complement:
indeed, there are explicit examples of arrangements with isomorphic
intersection poset, whose complements are not homeomorphic (see
\cite{Rybnikov}). In the case of a complexified real hyperplane arrangement,
the homeomorphism type of the complement is determined instead
by the \emph{face poset} of the induced (regular CW) decomposition
$\call D(\scr A)$ of $\mat R^n$. 
\end{oss}

In general, the homotopy type of a complexified toric arrangement
cannot be described in terms of the face poset of the induced
decomposition of the compact torus. Indeed Moci
and Settepanella in \cite{mocisettepanella} characterize exactly the arrangements
for which this poset describes the homotopy type of $M(\scr A)$: these are
the arrangements $\scr A$ for which $\call D(\scr A)$ is a regular
cell-complex or, in the terminology of \cite{mocisettepanella}, {\em
  thick} arrangements.\\

In our take at this matter we would like to keep full generality and therefore
suggest to replace the poset of faces with the following more general object.

\begin{defi} Let $\scr A$ be a complexified toric
  arrangement. Then $\call F(\scr A)$ will denote the \emph{face
    category} of the complex $\call D(\scr A)$  (see Definition
\ref{def:D}).
\end{defi}

\begin{oss} Thick arrangements are precisely those arrangement
for which the face category $\call F(\scr A)$ is a poset. For such arrangements
the construction of the Salvetti complex in the affine case translates almost 
literally to the toric case (see \cite{mocisettepanella} for the
details). 

Our construction is more general in the sense that it does not assume
thickness and, moreover, in the thick case it specializes to the
complex considered by Moci and Settepanella.
\end{oss}

\subsection{Restriction}

The operation of passing to sub arrangements, while intuitive and elementary
in the case of hyperplane arrangements, needs some
careful consideration in the toric case.\\

Let $\Gamma$ be a subgroup of the lattice $\Lambda$. 
Then $T_\Gamma := \Hom_{\mat Z}(\Gamma, S^1)$
is a compact ($\mbox{rk}\,\Gamma$)-torus and the inclusion 
$i_\Gamma:\Gamma \to \Lambda$ induces a surjection $\pi_\Gamma:
T_\Lambda \to T_\Gamma$ given by restriction: $\pi_\Gamma (p)=p_{\vert
\Gamma}$ . 

\begin{defi}\label{defi:restriction}
  Given a subgroup $\Gamma \subseteq
\Lambda$ and an arrangement $\scr A$ in $T_\Lambda$, we define the arrangement
\begin{displaymath}
  \scr A_{\Gamma} = \{(\chi, a) \in \scr A: \chi \in \Gamma\}.
\end{displaymath}
\end{defi}

\begin{prop}
  The map $\pi_\Gamma:T_\Lambda \to T_\Gamma$ induces a \emph{cellular map}
  $\pi^{\textrm{cell}}_\Gamma:\call D(\scr A) \to \call D(\scr A_\Gamma)$.
\end{prop}
\begin{proof}
  We can choose a basis $x_1,\ldots , x_n$ for $\Lambda$ such that
  $\Gamma = \langle x_1^{k_1}, \ldots , x_l^{k_l} \rangle$. The
isomorphism $T_\Lambda \simeq \mathbb C^n$ is given by evaluation on
the chosen basis:
$p\mapsto (p(x_1),\ldots p(x_n))$.
Therefore the projection $\sx\mathbb C^*\dx^n \rightarrow \sx\mathbb C^*\dx^l $ is
given by the map $(y_1, \ldots, y_n) \mapsto (y_1^{k_1}, \ldots, y_l^{k_l})$. 
This map is continuous and maps hypersurfaces (of $\scr
A_{\Gamma}\subseteq \scr A$ in $\sx\mathbb C^*\dx^n$) onto hypersurfaces (of
$\scr A_\Gamma$ in $\sx\mathbb C^*\dx^l$), hence is cellular.
\end{proof}

The construction of $\scr A_\Gamma$ is to be thought of as the analogue
of the quotient construction in \eqref{eq:quotienthyperplanes}. In
particular, given any face $F\in\call F (\scr A)$ we can let $\Gamma$
be the lattice 
\begin{displaymath}
\Lambda_{F}:=\{\chi\in\Lambda \mid \chi \textrm{ is constant on } F\}.
\end{displaymath}
Correspondingly, we obtain a toric subarrangement with an associated
cellular map:
\begin{equation}\label{eq:terminitorici}
\scr A_F:=\scr A_{\Lambda_F},\quad\quad 
\pi_F:=\pi_{\Lambda_F}^{\textrm{cell}}: \, \call{D}(\scr A) \rightarrow \call D(\scr A_F).
\end{equation}
The fact that $\pi_F$ is cellular implies that $\pi_F$ induces a
morphism of acyclic categories $\pi_F: \call F(\scr A) \rightarrow
\call F(\scr A_F)$.

\subsection{Covering spaces}\label{covering_spaces}

In order to connect the theory of toric arrangements to that of hyperplane
arrangements, we will look at a particular \emph{covering space} of 
a toric arrangament complement. Again, for our purposes it is convenient
to work with abstract tori.

Consider the following covering map
\begin{displaymath}
  \begin{array}{c}
    p: \Hom_{\mat Z}(\Lambda, \mat C) \to \Hom_{\mat Z}(\Lambda, \mat C^*)\\
    \varphi \mapsto \mbox{exp}\circ\varphi
  \end{array}
\end{displaymath}
where $\exp:\mat C \to \mat C^*$ is the exponential map, i.e., $ z \mapsto e^{2\pi i z}$. Notice
that $\Hom_{\mat Z}(\Lambda, \mat C) \cong \mat C^n$ and, through
this isomorphism, $p$ is just the universal covering map
\begin{displaymath}
  (t_1, \dots, t_n) \mapsto (e^{2\pi i t_1}, \dots, e^{2\pi i t_n})
\end{displaymath}
of the torus $T_\Lambda$.
Furthermore, $p$ restricts to a universal covering map 
$$\mat R^n \cong
\Hom_{\mat Z}(\Lambda, \mat R) \to \Hom_{\mat Z}(\Lambda, S^1) \cong
(S^1)^n$$ 
of the compact torus, under which the
preimage of a toric arrangement $\scr A$ is the (infinite)
affine hyperplane arrangement
\begin{displaymath}
  \tilde{\scr A} = \{(\chi,a') \in \Lambda \times \mat R
\mid (\chi,e^{2\pi i  a'}) \in \scr A\},
\end{displaymath}
or, in coordinates:
\begin{displaymath}
  \tilde{\scr A} = \{\gen{\alpha, x} =  a' \mid
  (x^\alpha, e^{2\pi i a'}) \in \scr A\}.
\end{displaymath}
Here $\alpha \in \mat Z^n$ and $x^\alpha$ is the associated character
$x_1^{\alpha_1} \cdots x_n^{\alpha_n}$. With this definition $p$ induces
a cellular map $p: \call D(\tilde{\scr A}) \to \call D(\scr A)$.\\

The arrangement $\tilde{\scr A}$ is a locally finite complexified
affine hyperplane arrangement and therefore admits a Salvetti complex
$$\tilde{\call S} = \tilde{\call S}(\scr A):=\call S(\tilde{\scr A}).$$ 
The character lattice $\Lambda$ acts cellulary on $\tilde{\call S}$
and continously on the covering space $M(\scr A)$. These two actions
are compatible, meaning that the embedding $\tilde{\call S} \to M(\tilde{\scr A})$ 
constructed in \cite{salvetti1987tcr} is $\Lambda$-equivariant (more precisely,
it can be so constructed).

\begin{figure}[tp]
  \centering
  \begin{tikzpicture}[scale = 1]
    \tikzstyle{every path} =
    [gray, line width = 1.5pt];
    \tikzstyle{every node} = 
    [circle, fill, inner sep=2pt, outer sep=3pt, draw];

    \draw[step = 2, thin] (-2,-2) grid (4,4);

    \draw node (A) at (0,1) {} node (B) at (1,0) {}
          node (C) at (2,1) {} node (D) at (1,2) {}
          node (E) at (-1,2) {} node (F) at (0,3) {}
          node (G) at (0,-1) {} node (H) at (-1,0) {}
          node (I) at (2,-1) {} node (L) at (3,0) {}
          node (J) at (3,2) {} node (K) at (2,3) {};
    
    \draw[->] (A) .. controls (0,.5) and (.5,0) .. (B);
    \draw[->] (D) .. controls (1,1.5) and (1.5,1) .. (C);
    \draw[->] (B) .. controls (1,.5) and (1.5,1) .. (C);
    \draw[->] (A) .. controls (0,1.5) and (.5,2) .. (D);
    \draw[->] (E) .. controls (-1,1.5) and (-.5,1) .. (A);
    \draw[->] (F) .. controls (0,2.5) and (.5,2) .. (D);
    \draw[->] (E) .. controls (-1,2.5) and (-.5,3) .. (F);
    \draw[->] (G) .. controls (0,-.5) and (.5,0) .. (B);
    \draw[->] (H) .. controls (-1,-.5) and (-.5,-1) .. (G);
    \draw[->] (H) .. controls (-1,.5) and (-.5,1) .. (A);
    \draw[->] (B) .. controls (1,-.5) and (1.5,-1) .. (I);
    \draw[->] (C) .. controls (2,.5) and (2.5,0) .. (L);
    \draw[->] (I) .. controls (2,-.5) and (2.5,0) .. (L);
    \draw[->] (C) .. controls (2,1.5) and (2.5,2) .. (J);
    \draw[->] (D) .. controls (1,2.5) and (1.5,3) .. (K);
    \draw[->] (K) .. controls (2,2.5) and (2.5,2) .. (J);

    \path [draw, white, line width = 10pt]
    (-2,-2) -- (4,4)   (-2, 0) -- (2,4) 
    (-2, 2) -- (0,4)   ( 0,-2) -- (4,2)
    ( 2,-2) -- (4,0)   ( 0,-2) -- (-2, 0)
    ( 2,-2) -- (-2, 2) ( 4,-2) -- (-2, 4)
    ( 4, 0) -- ( 0, 4) ( 4, 2) -- (2,4);

    \path [draw, red] (-2,-2) -- (4,4)
                      (-2, 0) -- (2,4) 
                      (-2, 2) -- (0,4)
                      ( 0,-2) -- (4,2)
                      ( 2,-2) -- (4,0);

    \path [draw, color=purple] ( 0,-2) -- (-2, 0)
                        ( 2,-2) -- (-2, 2)
                        ( 4,-2) -- (-2, 4)
                        ( 4, 0) -- ( 0, 4)
                        ( 4, 2) -- (2,4);

    \path [draw, white, line width = 10pt] 
    (B) .. controls (1,.5) and (.5,1) .. (A)
    (C) .. controls (2,1.5) and (1.5,2) .. (D)
    (C) .. controls (2,.5) and (1.5,0) .. (B)
    (D) .. controls (1,1.5) and (.5,1) .. (A)
    (A) .. controls (0,1.5) and (-.5,2) .. (E)
    (D) .. controls (1,2.5) and (.5,3) .. (F)
    (F) .. controls (0,2.5) and (-.5,2) .. (E)
    (B) .. controls (1,-.5) and (.5,-1) .. (G)
    (G) .. controls (0,-.5) and (-.5,0) .. (H)
    (A) .. controls (0,.5) and (-.5,0) .. (H)
    (I) .. controls (2,-.5) and (1.5,0) .. (B)
    (L) .. controls (3,.5) and (2.5,1) .. (C)
    (L) .. controls (3,-.5) and (2.5,-1) .. (I)
    (J) .. controls (3,1.5) and (2.5,1) .. (C)
    (K) .. controls (2,2.5) and (1.5,2) .. (D)
    (J) .. controls (3,2.5) and (2.5,3) .. (K);

    \draw[->] (B) .. controls (1,.5) and (.5,1) .. (A);
    \draw[->] (C) .. controls (2,1.5) and (1.5,2) .. (D);
    \draw[->] (C) .. controls (2,.5) and (1.5,0) .. (B);
    \draw[->] (D) .. controls (1,1.5) and (.5,1) .. (A);
    \draw[->] (A) .. controls (0,1.5) and (-.5,2) .. (E);
    \draw[->] (D) .. controls (1,2.5) and (.5,3) .. (F);
    \draw[->] (F) .. controls (0,2.5) and (-.5,2) .. (E);
    \draw[->] (B) .. controls (1,-.5) and (.5,-1) .. (G);
    \draw[->] (G) .. controls (0,-.5) and (-.5,0) .. (H);
    \draw[->] (A) .. controls (0,.5) and (-.5,0) .. (H);
    \draw[->] (I) .. controls (2,-.5) and (1.5,0) .. (B);
    \draw[->] (L) .. controls (3,.5) and (2.5,1) .. (C);
    \draw[->] (L) .. controls (3,-.5) and (2.5,-1) .. (I);
    \draw[->] (J) .. controls (3,1.5) and (2.5,1) .. (C);
    \draw[->] (K) .. controls (2,2.5) and (1.5,2) .. (D);
    \draw[->] (J) .. controls (3,2.5) and (2.5,3) .. (K);

    \path [fill = green, opacity = .5] (-1,0) 
    .. controls (-1,.5) and (-.5,1) .. (0,1) 
    .. controls (0,.5) and (.5,0) .. (1,0) 
    .. controls (.5,0) and (0,-.5) .. (0,-1)
    .. controls (-.5,-1) and (-1,-.5) .. (-1,0);

    \path [fill = green, opacity = .5] (-1,2) 
    .. controls (-1,2.5) and (-.5,3) .. (0,3) 
    .. controls (0,2.5) and (.5,2) .. (1,2) 
    .. controls (.5,2) and (0,1.5) .. (0,1)
    .. controls (-.5,1) and (-1,1.5) .. (-1,2);

    \path [fill = green, opacity = .5] (1,2) 
    .. controls (1,2.5) and (1.5,3) .. (2,3) 
    .. controls (2,2.5) and (2.5,2) .. (3,2) 
    .. controls (2.5,2) and (2,1.5) .. (2,1)
    .. controls (1.5,1) and (1,1.5) .. (1,2);

    \path [fill = green, opacity = .5] (1,0) 
    .. controls (1,.5) and (1.5,1) .. (2,1) 
    .. controls (2,.5) and (2.5,0) .. (3,0) 
    .. controls (2.5,0) and (2,-.5) .. (2,-1)
    .. controls (1.5,-1) and (1,-.5) .. (1,0);

    \path [fill = blue, opacity = .5] (1,0) 
    .. controls (1,.5) and (1.5,1) .. (2,1) 
    .. controls (2,1.5) and (1.5,2) .. (1,2) 
    .. controls (.5,2) and (0,1.5) .. (0,1)
    .. controls (.5,1) and (1,.5) .. (1,0);

  \end{tikzpicture}
  \caption{Salvetti Complex for $\tilde{\scr A}$}
  \label{fig:salvetticover}
\end{figure}

\begin{es}
  Figure \ref{fig:salvetticover} shows the Salvetti complex for
  the arrangement $\tilde{\scr A}$, with $\scr A = \{(ts,1), (ts^{-1},1)\}$.
  The green cells belong to the same $\Lambda$-orbit.
\end{es}

With the previous constructions in mind, we can now restate a key
result of \cite{mocisettepanella}.

\begin{prop}[{\cite[Lemma 1.1]{mocisettepanella}}]
  \label{prop:embeddingquoziente}
  Let $\scr A$ be an essential toric arrangement;
  the embedding $\tilde{\call S} \to M(\tilde{\scr A})$ induces
  an embedding
  \begin{displaymath}
    \tilde{\call S}/\Lambda \to M(\scr A)
  \end{displaymath}
  of the quotient $\tilde{\call S}$ in the complement $M(\scr A)$
  as a deformation retract.
\end{prop}

\begin{oss}
  In the proof of Proposition \ref{prop:embeddingquoziente} given in \cite{mocisettepanella} the hypotesis  of essentiality is  required. Indeed the construction of the homotopy
  inverse $\psi:\tilde{\call S}/\Lambda \to M(\scr A)$ does not work
  for non-essential arrangements.
\end{oss}

\section{Toric Salvetti complex}
\label{sec:salvetti}

We now head towards the first main theorem of this paper, introducing 
the notion of Salvetti complex for general complexified toric
arrangements with a construction that specializes to the complex of
\cite{mocisettepanella} in the case of thick arrangements.

\begin{defi}[Salvetti category]\label{df:salvetticategory}
  Let $\scr A$ be a toric arrangement on $(\mat C^*)^n$. The
  \emph{Salvetti Category} of $\scr A$ is the acyclic category
  $\zeta = \zeta(\scr A)$ defined as follows:
  \begin{enumerate}[(i)]
  \item the objects are the morphisms in $\call F(\scr A)$ between faces and
    chambers
    \begin{displaymath}
      \call O(\zeta) = \{m: F \to C:\quad m \in \call M(\call F(\scr A)),
      \, C \mbox{ chamber}\};
    \end{displaymath}
  \item for every morphism $n: F_2 \to F_1$ in $\call F(\scr A)$,
    and for every pair $m_1:F_1 \to C_1$, $m_2:F_2 \to C_2$ in $\call O(\zeta)$
    there is a morphism $(n,m_1,m_2): m_1 \to m_2$ if and only if
    \begin{equation}\label{eq:condition}
      \pi_{F_1}(m_1) = \pi_{F_1}(m_2);
    \end{equation}
    where $\pi_{F_1}$ is the morphism of face categories induced by the
    cellular map in \eqref{eq:terminitorici};
  \item let $m_i:F_i \to C_i$ for $i = 1,2,3$ be elements in $\call O(\zeta)$,
    suppose the pairs $(m_1, m_2)$ and $(m_1, m_3)$ satisfy condition
    \eqref{eq:condition}, then the pair $(m_1, m_3)$ satisfies the same condition
    and we can define for morphisms $n:F_2 \to F_1$, $n':F_3 \to F_2$
    the composition
    \begin{displaymath}
      (n',m_2,m_3)\circ(n,m_1,m_2) = (n\circ n',m_1,m_3).
    \end{displaymath}
  \end{enumerate}
\end{defi}

\begin{defi}
  Let $\scr A$ be a toric arrangement; its \emph{Salvetti complex}
  is the nerve $\Delta(\zeta(\scr A))$.
\end{defi}

We can now state the main theorem of this section.

\begin{thmA}
\label{teo:teoremone}
  Let $\Lambda$ be a lattice and $\scr A$ be a complexified toric
  arrangement in $T_\Lambda$. The nerve
  $\Delta(\zeta(\scr A))$ embeds in $M(\scr A)$ as a deformation retract.
\end{thmA}

\begin{oss}
  Being the nerve of an acyclic category, $\Delta(\zeta(\scr A))$ is a 
  regular trisp.
\end{oss}

\begin{oss} In the case of affine arrangements of hyperplanes, the Salvetti poset defined in Section \ref{sec:hyperplanes} is indeed the poset of
cells of a regular CW-complex, of which the (simplicial) Salvetti
complex is the barycentric subdivision. Earlier we have called this
the ``unsubdivided''  Salvetti complex. Our goal now is to describe a CW complex of which the nerve
$\Delta(\zeta)$ is the barycentric subdivision. This complex will not
be regular in general, but the resulting economy in terms of cells
will come in handy in the following considerations.

Let then $\scr A$ denote a toric arrangement. Every cell of the
unsubdivided Salvetti complex of $\tilde{\scr A}$ corresponds to the
topological closure of the star of a vertex $[F,C]$ of the subdivided
complex.  Because the projection $\mbox{Sal}(\tilde{\scr A})\to \zeta$
is a covering of categories, the interior of the star of any vertex of
the nerve $\Delta(\mbox{Sal}(\tilde{\scr A}))$ is mapped
homeomorphically to the interior of the star of its image. This gives
a canonical CW-structure on $\Delta(\zeta)$. 
The acyclic category $\zeta$ is precisely the
face category of the resulting CW complex.

In particular, the explicit determination of the boundary maps of this
complex is now reduced to a straightforward computation.
\end{oss}

Before we can get to the proof, some preparatory considerations are in order.

\subsection{Restriction vs. covering}

In order to proceed with the argument we still need to spend a few words
on the quotient construction of \eqref{eq:quotienthyperplanes} and its
toric analogue.

Let $F$ be a face of $\call D(\scr A)$ and let $\Lambda_F$ be the
sublattice of characters in $\Lambda$ that are constant on $F$. Every
$\varphi\in \Lambda_F$ is then constant on the affine subspace spanned
by $F$, which we write $y+ L$ for $y\in \mathbb R^n$ and $L$ a linear
subspace of $\mathbb R^n$: therefore $\varphi$ vanishes on $L$.
Then we have an isomorphism
\begin{equation}\label{eq:isoquoziente}
 \rho: \,\, \mat R^n/L \to \Hom_{\mat Z}(\Lambda_F, \mat R).
\end{equation}
Recall from \eqref{eq:terminitorici} the arrangement
\begin{displaymath}
  \scr A_F = \{(\chi,a) \in \scr A:\, \chi \in \Lambda_F\} \subseteq
  \scr A
\end{displaymath}
in $\Hom_{\mat Z}(\Lambda_F, \mat R)$. The isomorphism $\rho$ from 
\eqref{eq:isoquoziente} does not map the arrangement 
$(\tilde{\scr A})_F$ onto $\tilde{(\scr A_F)}$. Indeed
$\tilde{(\scr A_F)}$ contains all the translates of the hyperplanes in $(\tilde{\scr A})_F$. That is
\begin{displaymath}
(\tilde{\scr A})_F \subseteq  \tilde{\scr A_F} = \{(\chi, a + k) \mid (\chi, a) \in 
({\scr A}^\upharpoonright)_F,\, k \in \mat Z\}
\end{displaymath}
and therefore we have a natural cellular support map
$$
s: \call D({\tilde{\scr A_F}}) \to  \call D(\tilde{\scr A}_F)
$$
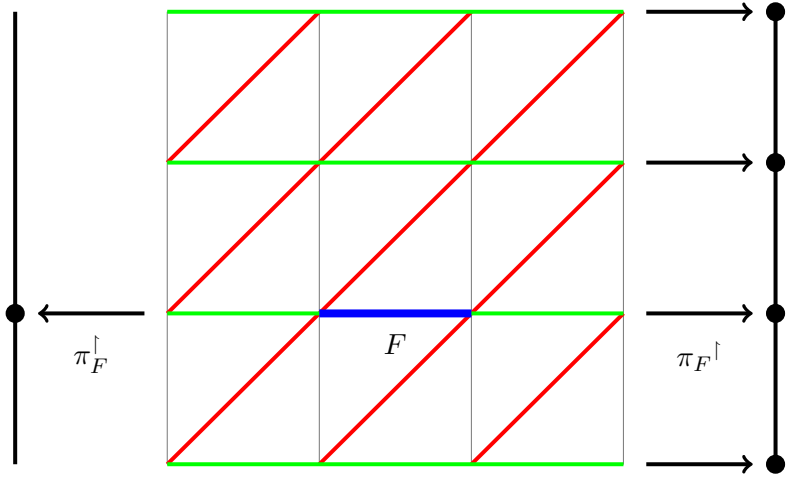
\begin{figure}[tp]
  \centering
  \begin{tikzpicture}[scale = 1]
    \tikzstyle{every path} =
    [black, line width = 1.5pt];
    \tikzstyle{every node} = 
    [circle, fill, inner sep=2pt, outer sep=3pt, draw];

    \draw[step = 2, gray, thin] (-2,-2) grid (4,4);

    \path [draw, red] (-2,-2) -- ( 4, 4)    (-2, 0) -- ( 2, 4) 
                      (-2, 2) -- ( 0, 4)    ( 0,-2) -- ( 4, 2)
                      ( 2,-2) -- ( 4, 0);

    \path [draw, green] (-2,-2) -- ( 4,-2)   (-2, 0) -- ( 4, 0)
                       (-2, 2) -- ( 4, 2)   (-2, 4) -- ( 4, 4); 

    \path [draw, blue, line width = 3pt] ( 0, 0) -- 
    node [draw = none, fill = none, black, below] {$F$} ( 2, 0);

    \path node (A) at ( 6,-2) {}    node (B) at ( 6, 0) {} 
          node (C) at ( 6, 2) {}    node (D) at ( 6, 4) {};
    \path [draw] (6,-2) -- (6,4);
      
    \draw [->] (4.3, 4) -- (5.7, 4);  \draw [->] (4.3,2) -- (5.7,2);
    \draw [->] (4.3,0) -- 
    node [draw = none, fill = none, below] {$\tilde{\pi_{F}}$} (5.7,0);
    \draw [->] (4.3,-2) -- (5.7,-2);

    \path node (E) at (-4,0) {};
    \path [draw] (-4,-2) -- (-4,4);
      
    \draw [->] (-2.3,0) -- 
    node [draw = none, fill = none, below] {$\tilde{\pi}_{F}$} (-3.7,0);
  \end{tikzpicture}
  \caption{Restriction vs. Covering}
  \label{fig:restrictions}
\end{figure}

The map $\pi_F$ of \eqref{eq:terminitorici} lifts (via $p$) to a map $\mathbb R ^{rk
  \Lambda} \rightarrow \mathbb R^{rk \Lambda_F}$ which induces a cellular map
$
\tilde{\pi_F}:\, \call D(\tilde{\scr A})\to \call D(\tilde{(\scr A_F)})
$
and the following diagram commutes
\begin{equation}\label{diagrammacellulare}
    \xymatrix{
      \call D(\tilde{\scr A}) \ar[r]^{\tilde{\pi_{F}}} \ar[d]_p &
      \call D(\tilde{(\scr{A}_F)}) \ar[d]^p\\
      \call D(\scr A) \ar[r]_{\pi_F} & 
      \call D(\scr A_F)
    }
  \end{equation}

On the other hand, in $\Hom(\Lambda, \mathbb R)$ we have the
projection from \eqref{eq:salvettiaffine}, which we call
$\tilde{\pi}_F$ and in terms of
which the Salvetti complex of $\tilde{\scr A}$ is defined, which is
$$
\tilde{\pi}_F: \call D(\tilde{\scr A}) \to \call D((\tilde{\scr A})_F)
$$
and is related to $\tilde{\pi_F}$ via
$$\tilde{\pi}_F = s\circ \tilde{\pi_F}.$$

Figure \ref{fig:restrictions} shows an example of projections
$\tilde{\pi}_F$ and $\tilde{\pi_F}$.

\begin{lemma} Let $F_1,F_2,C_1,C_2\in\call F(\tilde{\scr A})$ with
  $C_1, C_2$ chambers, $F_1\leq C_1$ and $F_1\leq F_2 \leq C_2$. Then
$$
\tilde{\pi_{F_1}}(C_1) = \tilde{\pi_{F_1}}(C_2) \iff \tilde{\pi}_{F_1}(C_1)=\tilde{\pi}_{F_1}(C_2).
$$
\end{lemma}

\begin{proof} The direction $\Rightarrow$ follows since $\tilde{\pi}_F = s\circ \tilde{\pi_F}$. For $\Leftarrow$: if $\tilde{\pi}_{F_1}(C_1) = \tilde{\pi}_{F_1}(C_2)$, then
$\tilde{\pi_{F_1}}(C_1) = \tilde{\pi_{F_1}}(C_2 + \lambda)$, for
some $\lambda \in \Lambda_F$. But since
$F_2$ is a common face of $C_1$ and $C_2$, it has to be
$\lambda = 0$.
\end{proof}

\begin{cor}\label{cor:rel} Let $[F_1, C_1]$, $[F_2, C_2]$ denote two
  elements of $\mbox{Sal}\,\tilde{\scr A}$, 
  the Salvetti poset of $\tilde{\scr A}$. Then
  $$
  [F_1, C_1] \leq [F_2, C_2]  \iff F_1 \geq F_2 \textrm{ in } 
  \mathcal{F}(\scr A) \textrm{ and } \tilde{\pi_{F_1}}(C_1) = \tilde{\pi_{F_1}}(C_2) 
  $$
\end{cor}

\subsection{Quotients}

Our strategy for the proof of Theorem \ref{teo:teoremone} will be
to prove that the toric Salvetti  complex $\Delta(\zeta)$ is the
quotient of the action $\Lambda \agisce \tilde{\call S}$ in the
category of trisps. For this, we need first to take care of some
ground work.

\begin{lemma}\label{lemma:covering1}
  Let $\scr A$ be a complexified toric arrangement. Then  there is a covering $q:\call F(\tilde{\scr A}) \to \call F(\scr A)$
  of acyclic categories with Galois group $\Lambda$ and
$$\call F(\scr A) = \call F(\tilde{\scr A})/\Lambda$$  as a quotient of
acyclic categories.
\end{lemma}
\begin{proof}
  Let $F \in \call D(\tilde{\scr A})$ be a face of the affine arrangement
  $\tilde{\scr A}$. In particular $F$ is a polytope and
  $p(F) \in \call D(\scr A)$ is a face of $\scr A$. We
  can then use $F$ a polytopal model of $p(F)$ in Definition \ref{def:facecat}
  and map a morphism $F' \leq F$ to the corresponding morphism
  $m_{F',F}$.
  
  This defines a functor $q:\call F(\tilde{\scr A}) \to \call F(\scr A)$.
  Furthermore $q$ is a covering of categories in the sense of
  \cite[Definition A.15]{bridson1999metric} with $\Lambda$ as
  automorphism group and $\Lambda$ acts transitively on the fibers of
  $q$. It then follows\ifthenelse{\boolean{appendice}}{(see Proposition
  \ref{prop:rivestimentoquoziente} in the Appendix)}{}
  that $\call F(\scr A) \cong \call F(\tilde{\scr A})/\Lambda$.
\end{proof}

In particular, we note the following consequence.

\begin{cor}
 The morphisms in $\call F(\scr A)$ correspond to the orbits\\
  $\{\Lambda(F_1 \leq F_2) \mid \, F_1, F_2 \in \call D(\tilde{\scr A})\}.$
\end{cor}

Now we can prove a key lemma, finally making sense of our definition of $\zeta$.

\begin{lemma}
  The category $\zeta$ is the quotient $\mbox{Sal}(\tilde{\scr A})/\Lambda$
  in the category of acyclic categories.
\end{lemma}
\begin{proof}
  We first need to construct a projection, i.e., a functor 
  $\Pi:\mbox{Sal}\,(\tilde{\scr A}) \to \zeta$. 
  Recall that the objects of $\mbox{Sal}\,(\tilde{\scr A})$ are
  of the form $[F, C]$ with $F,C\in \call F (\tilde{\scr A})$, $F \leq
  C$, and $C$ a chamber of $\tilde{\scr A}$. 
Also, from the proof of Lemma \ref{lemma:covering1} we recall the
projection $q:\call F(\tilde{\scr A })\to \call F(\scr A)$.
It is now possible to define $\Pi$ on the objects as follows:
  \begin{displaymath}
    \Pi([F, C]) = q(F \leq C):q(F) \to q(C).
  \end{displaymath}
  
 According to Corollary \ref{cor:rel}, relations in $\call F(\tilde{\scr A})$ are of the form
  $[F_1, C_1] \leq [F_2, C_2]$ where $F_2 \leq F_1$ and $
   \tilde{ \pi_{F_1}}(C_1) =\tilde{\pi_{F_1}}(C_2).$

  On the other hand, morphisms in $\zeta(\scr A)$ are given by triples
  $(n, m_1, m_2)$ where $m_1:F_1 \to C_2$, $m_2:F_2 \to C_2$ are objects of $\zeta$,
  $n: F_2 \to F_1$ is a morphism in $\call F(\scr A)$ and the following
  condition holds:
  \begin{displaymath}
    \pi_{F_1}(m_1) = \pi_{F_1}(m_2).
  \end{displaymath}

  Therefore, in order to able to map a relation
  $[F_1, C_1] \leq [F_2, C_2]$ to the morphism
  $(q(F_2 \leq F_1), \Pi([F_1, C_1]), \Pi([F_2, C_2]))$ and for this
  map to be surjective, we need to verify the following condition:
  \begin{displaymath}
    \tilde{\pi_{F_1}}(C_1) = \tilde{\pi_{F_1}}(C_2) \iff
    \pi_{q(F_1)}(\Pi([F_1, C_1])) = \pi_{q(F_1)}(\Pi([F_2, C_2])).
  \end{displaymath}
We go back to the diagram (\ref{diagrammacellulare}), and write the
corresponding commutative diagram of face categories:
  \begin{displaymath}
    \xymatrix{
      \call F(\tilde{\scr A}) \ar[r]^{\tilde{\pi_{F_1}}} \ar[d]_q &
      \call F(\tilde{\scr A_{F_1}}) \ar[d]^q\\
      \call F(\scr A) \ar[r]_{\pi_{q(F_1)}} & F(\scr A_{q(F_1)})  }
  \end{displaymath}
  Now $\tilde{\pi_{F_1}}$ is a map of posets and since $\tilde{\pi_{F_1}}(F_1) = 
  \tilde{\pi_{F_1}}(F_2)$ we have
  \begin{displaymath}
    \tilde{\pi_{F_1}}(C_1) = \tilde{\pi_{F_1}}(C_2) \iff
    \tilde{\pi_{F_1}}(F_1 \leq C_1) = \tilde{\pi_{F_1}}(F_2 \leq C_2).
  \end{displaymath}
  Furthermore $q$ is a covering of categories, in particular
  is injective on the morphisms incident on $\tilde{\pi_{F_1}}(F_1)$.
  It then follows that
  \begin{align*}
    \tilde{\pi_{F_1}}(F_1 \leq C_1) = \tilde{\pi_{F_1}}(F_2 \leq C_2) \Leftrightarrow
    q\circ\tilde{\pi_{F_1}}(F_1 \leq C_1) = q\circ\tilde{\pi_{F_1}}(F_2 \leq C_2)\\
    \Leftrightarrow \pi_{q(F_1)}(q(F_1 \leq C_1)) = \pi_{q(F_1)}(q(F_2 \leq C_2)).
  \end{align*}

  Concluding: the functor $\Pi$ is well defined and 
  it now follows easily from Lemma \ref{lemma:covering1} that it is
  a Galois covering of acyclic categories with $\Lambda$ as automorphism
  group.
\end{proof}

We want to show that, in our particular case, the nerve
construction commutes with the quotient. Babson and Kozlov in \cite{Babson2005439}
give a necessary and sufficient condition for this:
\begin{prop}[{\cite[Theorem 3.4]{Babson2005439}}]
  \label{prop:commutaquoziente}
  Let $\call C$ be an acyclic category equipped with a group
  action $G \agisce \call C$. A canonical isomorphism
  $\Delta(\call C)/G \cong \Delta(\call C/G)$ exists if and only if the 
  following condition is satisfied:
  \begin{itemize}
  \item[] Let $t \geq 2$ and let $(m_1, \dots, m_{t-1}, m_a)$,
    $(m_1, \dots, m_{t-1}, m_b)$ composable morphism chains.
    Let $G m_a = G m_b$, then ther exists some $g \in G$, such that
    $g(m_a) = m_b$ and $g(m_i) = m_i$, $\forall i \in \{1, \dots, t-1\}$.
  \end{itemize}
\end{prop}

The next lemma ensures that we can apply the previous proposition to
our case.

\begin{lemma}\label{lemma:galoiscondition}
  Let $\call C$ be an acylic category and $G \agisce \call C$ act
  as the Galois group of a covering map. 
  Then the condition of proposition \ref{prop:commutaquoziente} is satisfied.
\end{lemma}
\begin{proof}
  Consider two composable morphism chains as in the condition of proposition
  \ref{prop:commutaquoziente}. Since $t \geq 2$ and the chains are composable,
  $m_a$ and $m_b$ must have the same domain, $m_a: p \to q$, $m_b:p \to r$.
  Furthermore there is a $g \in G$, such that $m_b = gm_a$.
  
  Let $\varphi: \call C \to \call D$ be a covering map with Galois group $G$.
  Then $\varphi(m_a) = \varphi(m_b) \Rightarrow m_a = m_b$ and the condition
  is trivially satisfied.
\end{proof}

We finally get to the proof of Theorem \ref{teo:teoremone}, which now
follows as an application of the previous considerations.

\begin{proof}[Proof of Theorem \ref{teo:teoremone}]
  According to proposition \ref{prop:embeddingquoziente} the statement holds
  for the complex $\tilde S/\Lambda = \Delta(\mbox{Sal}\,\tilde{\scr A})/\Lambda$.
  The lattice $\Lambda$ acts on $\tilde S$ as the automorphism group 
  of a covering map, in particular lemma \ref{lemma:galoiscondition} 
  holds and we have:
  \begin{displaymath}
    \tilde S/\Lambda = \Delta(\mbox{Sal}\,\tilde{\scr A})/\Lambda \cong
    \Delta(\mbox{Sal}\,\tilde{\scr A}/\Lambda) \cong \Delta(\zeta).
  \end{displaymath}
\end{proof}

%%%%%%%%%%%%%%%%%%%%%%%%%%%%%%%%%%%%%%%%%%
%%%%%%%%%%%%%%%%%%%%%%%%%%%%%%%%%%%%%%%%%%%%
%%%%%%%%%%%%%%%%%%%%%%%%%%%%%%%%%%%%%%%%%%%%%%

\section{The fundamental group}\label{sec:fg}

As an application of the results of the previous sections, and in a
structural tribute to the seminal paper of Salvetti
\cite{salvetti1987tcr}, we would like to give a presentation for the
fundamental group of a complexified toric arrangement.\\

\def\toro{T_\Lambda}
\def\fg{\pi_1}
\def\compl{M(\scr A)}

\subsection{Product structure}

First, note that the inclusion $M(\scr A) \to \toro$ induces an epimorphism of groups 
$$\epsilon: \fg(M(\scr A)) \to \fg(\toro) \simeq \mathbb Z^n. $$

\def\conv{\textup{conv}}
\begin{lemma}\label{lm:section} The map $\epsilon$ has a section $\xi$.
\end{lemma}
\begin{proof}

Choose a point $y\in \mathbb R^n$ in a chamber of $\tilde{\scr
  A}$. Then for all choices of $x\in \mathbb R^n$ we have
$$
x+iy\in M(\tilde{\scr A}).
$$ 
Accordingly,  for every choice of
arguments $\theta_1,\ldots ,\theta_n\in \mathbb R$, 
$$(\lambda_1e^{2\pi i \theta_1},\ldots , \lambda_n e^{2\pi i
  \theta_n})\in M(\scr A)$$

where, for all $j=1,\ldots ,n$, $\lambda_j:= e^{-2\pi y_j}$
This defines a map
$$f: T_\Lambda \to M(\scr A),\quad z\mapsto (\lambda_1e^{2\pi i
  \arg z_1},\ldots , \lambda_n e^{2\pi i \arg z_n}) $$
that induces a homomorphism $$\xi: \fg(T_\Lambda) \to \fg( M(\scr
A)).$$

Since $f$ is a homotopy (right-) inverse to the inclusion $M(\scr A) \to
T_\Lambda$, 
$\epsilon\xi = id$ and $\xi$ is
the required section.
\end{proof}

\begin{lemma}\label{lm:sequence}
The sequence
$$0\to p_*(\fg(\tilde{\call S})) \stackrel{\iota}{\to} \fg(\compl)
\stackrel{\epsilon}{\to} \fg(\toro)\to 0$$
is split exact. Therefore
$$
\fg(M(\scr A)) \simeq \fg(\tilde{\call S}) \rtimes \fg(\toro).
$$
\end{lemma}

\begin{proof}
We already showed that the map $\epsilon$ has a section, we then need
only to prove $\iota(p_*(\fg(\tilde{\call S}))) = \Ker\,\epsilon$.
It is clear that $\iota(p_*(\fg(\tilde{\call S}))) \subseteq
\Ker\,\epsilon$.
For the opposite inclusion we consider the sequence
$$0\to p_*(\fg(M(\tilde{\scr A})))\to \fg(\compl)\to \fg
(\toro)\to 0$$
Let $[\gamma] \in \fg(M(\scr A))$ be an element of $\Ker\,\epsilon$.
Let $j$ be the inclusion of $M(\scr A)$ in the ambient torus $\toro$. Then $j\circ\gamma$ is a null homotopic loop in $T_\Lambda$ and
lifts therefore to a closed path $\gamma'$ in the universal cover $\mat C^n$.
Let $\tilde\gamma$ be the lift of $\gamma$ to $M(\tilde{\scr A})$
with base point $x$, then $\gamma' = \tilde{j}\circ\tilde\gamma$ and 
$\tilde\gamma$ is also a closed path. 
That is, $[\gamma] = p_*[\tilde\gamma] \in p_*(\fg(M(\tilde{\scr A}))) \cong
p_*(\fg(\tilde{\call S}))$.
\end{proof}

\subsection{Presentation of $\fg(M(\tilde{\scr A}))$}\label{subsec:gen}

As a stepping stone towards the computation of a presentation for the fundamental group of
$M(\scr A)$, we establish some notation and recall the
presentation of $\fg(\tilde{\call S})$ given by Salvetti in
\cite{salvetti1987tcr}.\\

Choose - and from now fix - a chamber $C_0$ of $\tilde{\scr A}$, and
let $x_0$ be a generic point in $C_0$ - i.e. such that for all
$i=1,\ldots , d$ the
straight line segment $s_i$
from $x_0$ to $u_ix_0$ meets only faces of codimension at most
$1$.

\begin{oss}\label{oss:codimensione}
In general, given a set $\call K$ of cells of a complex, $\call
K_i$ will denote the subset of cells of codimension $i$. 

Also, to streamline notation we will from now write $\call F$,
respectively $\tilde{\call F}$ for $\call F(\scr A)$, $\call
F(\tilde{\scr A})$.
\end{oss}

\newcounter{conto}
\addtocounter{conto}{1}
%\noindent {\em \thelemma.\Roman{conto}. Generators.}  

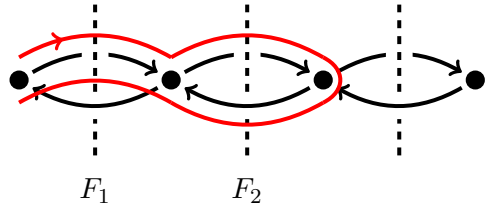
\begin{figure}[thp]
  \centering
    \begin{tikzpicture}[scale = 1]
    \tikzstyle{every path} =
    [black, line width = 1.5pt];
    \tikzstyle{every node} = 
    [circle, fill, inner sep=2pt, outer sep=3pt, draw];
    
    \path [draw] node (A) at (0,0) {}
    node (B) at (2,0) {}    node (C) at (4,0) {}
    node (D) at (6,0){};

    \path[->] (A) edge [bend left = 30] (B)
    (B) edge [bend left = 30] (C)
    (C) edge [bend left = 30] (D);
 
    \draw[white, line width = 10pt] (1,-1) -- (1,1) (3,-1) -- (3,1) (5,-1) -- (5,1);    
    \draw[dashed] (1,-1) -- (1,1) (3,-1) -- (3,1) (5,-1) -- (5,1);

    \draw (D) edge [white, line width = 10pt, bend left = 30] (C)
    (C) edge [white, line width = 10pt, bend left = 30] (B)
    (B) edge [white, line width = 10pt, bend left = 30] (A);
    
    \path[->] (D) edge [bend left = 30] (C)
    (C) edge [bend left = 30] (B)
    (B) edge [bend left = 30] (A);
    
    \path node [draw = none, fill = none, below] (F1) at (1,-1) {$F_1$}
    node [draw = none, fill = none, below] (F1) at (3,-1) {$F_2$};
    
    \draw[decoration={markings,mark=at position 0.3 with {\arrow{>}}}]
    (0,.3) edge [red, bend left = 30, postaction={decorate}] (2,.3)
    (2,.3) edge [red, bend left = 30] (4,.3) 
    (4,-.3) edge [red, bend left = 30] (2,-.3)
    (2,-.3) edge [red, bend right = 30] (0,-.3);
    \draw[red] (4,.3) .. controls (4.3,.15) and (4.3,-.15).. (4,-.3);
    \end{tikzpicture}
    \caption{Generators, an example: $\beta_{F_2} = l_{F_1}l_{F_2}^2l_{F_1}^{-1}$}
\end{figure}

\paragraph{\thesubsection.\theconto. Generators.}
Recall the graph $\tilde{\call G} := \call G(\tilde{\scr A})$ 
of Definition \ref{def:unsubdivided}. Here we will adopt a useful
notational convention inspired by \cite{salvetti1987tcr}: we will
write edges of $\tilde{\call G}$ as indexed by the  face of
codimension 1 they cross, and in writing a path we will write $l_F$
for a crossing of $F$ `along the direction of the edge', $l_F^{-1}$
for a crossing `against the direction' of the edge. By specifying the
first vertex of the path then there is no confusion about which edge is
used, and in which direction.

A {\em positive} path then is a path of the form 
$$l_{F_1}l_{F_2}\ldots l_{F_k}$$
for $F_1,\ldots F_k\in \tilde{\call F}_1$. It is also {\em minimal} if
the hyperplane supporting $F_i$ is different from the hyperplane
supporting $F_j$ for all $i\neq j$.

Since any two positive minimal paths with same origin and same end are
homotopic, given $C,C'\in\tilde{\call F}_0$ we will sometimes write
$(C\to C')$ for the (class of) positive minimal paths starting at $C$
and ending at $C'$.

For every $F\in\tilde{\call F}_1$ we define a path as follows:

\begin{equation}\label{eq:beta}
\beta_F:= (C_0\to (C_0)_F)l_F^2 (C_0\to (C_0)_F)^{-1},
\end{equation}

where, here and in the following, for a chamber $C$ and a face $F$ the
expression $C_F$ will denote the unique chamber in $\pi_F^{-1}(\pi_F(C))$ that
contains $F$ in its boundary.

\begin{lemma}[p. 616 of \cite{salvetti1987tcr}]
The group $\fg(\tilde{\call S})$ is
generated by the set $\{\beta_F\vert F\in \tilde{\scr F}_1 \}$.
\end{lemma}

Given a positive path $\nu=l_{F_1},\ldots,l_{F_k}$ define loops
\begin{equation}\label{eq:nu}
\beta^\nu_{F_i}:=l_{F_1}\cdots  l_{F_{i-1}} l_{F_i}^2 l_{F_{i-1}}^{-1}\cdots l_{F_1}^{-1}.
\end{equation}
Moreover, let $F_{j_1},\ldots , F_{j_l}$ be the sequence
obtained from $F_1,\ldots , F_k$ by recursively deleting faces $F_j$
that are supported on a hyperplane which supports an odd number of
elements of $F_{j+1},\ldots , F_k$ (compare
\cite[p. 614]{salvetti1987tcr}) and define
\begin{equation}\label{eq:sigma}
\Sigma(\nu):=(F_{i_l}, \ldots , F_{i_l}).
\end{equation}

\begin{lemma}[Lemma 12 in \cite{salvetti1987tcr}]\label{lm:salv12}
Given a positive path $\nu=l_{F_1},\cdots , l_{F_k}$ starting in the chamber
$C$ and ending in $C'$.
Then there is a homotopy
$$
\nu \simeq \big(\prod_{G\in\Sigma(\nu)}\beta_G^\nu \big) (C\to C').
$$ 
\end{lemma}

From this Lemma another useful result follows.

\begin{lemma}[Corollary 12 in \cite{salvetti1987tcr}]\label{lemma:12}
Let $F$,$G$ be two faces of codimension $1$ that are supported on the
same hyperplane. Then $\beta_F$ is homotopic to 
$$ (\prod_{i=1}^h \beta^{\nu}_{j_i})\beta_G  (\prod_{i=1}^h
\beta^{\nu}_{j_i})^{-1}, $$
where $\nu$ is a positive minimal path from $C_0$ to $\pi_G(C_0)$, and
$j_1,\ldots, j_h$ are the indices of the edges in $\nu$ that cross a
hyperplane that does not separate $C_0$ from $\pi_F(C_0)$, in the
order in which they appear in $\nu$. 
\end{lemma}

\addtocounter{conto}{1}
\paragraph{\thesubsection.\theconto. Relations.}
For every face $G\in \tilde{\call F}_2$ consider a chamber $C>G$ and
let $C'$ be its opposite chamber with respect to $G$. Consider a
minimal positive path $\omega$ from $C$ to $C'$. Let us then
consider the set $h(G):=\{F_1,\ldots F_k\}$ of
the codimension 1 faces adjacent to $G$, indexed according to the order in
which the positive minimal path $\omega$ `crosses' them. This ordering
is well defined up to cyclic permutation. Let now for $i=1,\ldots k$
$F_{i+k}$ be the facet opposite to $F_i$ with respect to $G$. Define a
path
\begin{equation}\label{eq:alfa}
\alpha_G(C):=l_{F_1}l_{F_2}\ldots l_{F_{2k}}.
\end{equation}

Salvetti introduces a set of relations associated with $G$:
$$R_G: \quad\beta_{F_1}\ldots \beta_{F_k} = \beta_{F_2}\ldots
\beta_{F_k}\beta_{F_1} =\ldots $$
stating the equality of all cyclic permutations of the product.
In fact, for every cyclic permutation $\sigma$ of $\{1,\ldots ,k\}$
\begin{equation}\label{eq:alfabeta}
\beta_{F_{\sigma(1)}} \cdots \beta_{F_{\sigma(k)}} \simeq (C_0\to
\widetilde{C})\alpha_G(\widetilde{C})(C_0\to \widetilde{C})^{-1}
\end{equation}
where $\widetilde{C}:=(C_0)_G$ and $\simeq$ means homotopy.

\addtocounter{conto}{1}
\paragraph{\thesubsection.\theconto. Presentation.}
 One of the results of \cite{salvetti1987tcr} is that the fundamental
group of $M(\tilde{\scr A})$ can be presented as
$$
\fg(\tilde{\call S})=\langle \beta_F,\, F\in \tilde{\call F}_1 \mid
R_G,\, G\in \tilde{\call F}_2 \rangle.
$$

\subsection{Generators}

We describe the action of $u\in\Lambda$ on a path
$\gamma\in\tilde{\call G}$ by writing $u.\gamma$ for the path obtained by translation of $\gamma$
with $u$.

\begin{defi}
Choose a basis $u_1,\ldots u_n$ of $\Lambda$, and for $i=1,\ldots d$
let $ \omega_i =\omega^{(1)}_i$ be the positive minimal path of $\tilde{\call G}$ from $C_0$ to
  $u_iC_0$ obtained by crossing the faces met by the straight line segment $s_i$
  (which connects from $x_0$ to $u_ix_0$). Also, for $k\geq 1$ let
  $\omega_i^{(k)}=\omega_i(u_i.\omega_i^{(k-1)})$. Similarly, let
  $\omega_i^{(-1)}:=\omega_i^{-1}$ and
  $\omega_i^{(-k)}:=\omega_i^{(-1)} (u_i^{-1}.\omega_i^{(1-k)})$.
Given any $u\in \Lambda$ write $u=u_1^{q_1}\cdots u_n^{q_n}$ and
define 

\begin{equation}
\omega_u:=\omega_1^{(q_1)}\; u_1^{q_1}.\omega_2^{(q_2)}\; \cdots
\big(\prod_{j=1}^{r-1}u_n^{q_n}\big).\omega_r^{(q_n)}. 
\end{equation}

Let then $$\tau_i:=p_*(\omega_i),\quad\quad\tau_u:=p_*(\omega_u).$$
\end{defi}

Notice that a path $\omega_u$ needs not be minimal, nor positive. In
fact, it is positive if and only if $u$ has nonnegative coordinates in
$\Lambda$. Given $i$ and $k$, the path $\omega_i^{(k)}$ is positive if and only if
$k\geq 0$, and in this case it is also minimal.

\begin{figure}[tp]
  \centering
  \begin{tikzpicture}[scale = 2]
    \tikzstyle{every path} =
    [black, line width = 1pt];
    \tikzstyle{every node} = 
    [circle, fill, inner sep=1pt, outer sep=2pt, draw];

    \path [fill = gray, opacity = 0.3] (-1,2) -- (-1,1) -- (0,0) -- (1,1)
    -- (2,0) -- (3,1) -- (3, 2) -- cycle;
    
    \draw (-1,0) -- (3,0);
    \draw [red] (-1,2) -- (3,2) (-.5,2) -- (-.5,.6);
    \draw [blue] (1,0) -- (1,2.5);
    
    \draw node (A) at (0,0) {}
    node (B) at (2,0) {} node (C) at (1,1) {};

    \draw [red] node (E) at (-.5,2) {}
    node (G) at (-.5,.6) {};
      
    \begin{scope}
      \tikzstyle{every node} = [draw = none, fill = none]

      \path node [below] (AA) at (0,0) {$s_j(t_i')$}
      node [below] (BB) at (2,0) {$s_j(t_{i+1}')$}
      node [below] (DD) at (1,0) {$s_j(t_i)$}
      node [above right] (CC) at (1,1) {$P_i$}
      node [above] (EE) at (-.5,2) {$s_j'(t)$}
      node [right] (FF) at (-.5,1.5) {$w(j,t)$}
      node [right] (GG) at (-.5,.6) {$r_j(t)$}
      node [blue, right] (EE) at (1,2.3) {$F_i + i (C_0)_{F_i}$}
      node [below right] (HH) at (.4,.5) {$r_{i,1}$}
      node [below left] (II) at (1.6,.5) {$r_{i,2}$};
    \end{scope}

    \path [draw]   
    (-1,1) -- (0,0) -- (1,1) -- (2,0) -- (3, 1);

    \draw[->] (0,0) -- (.7,.7);
    \draw[->] (1,1) -- (1.3,.7);
    
    \path [draw, red, rounded corners = 10pt]
    (-1,1.1) -- (0,0.1) -- (1,1.1) -- (2,0.1) -- (3, 1.1);
  \end{tikzpicture}
  \caption{Construction for the proof of Lemma \ref{lm:tau}}
\end{figure}
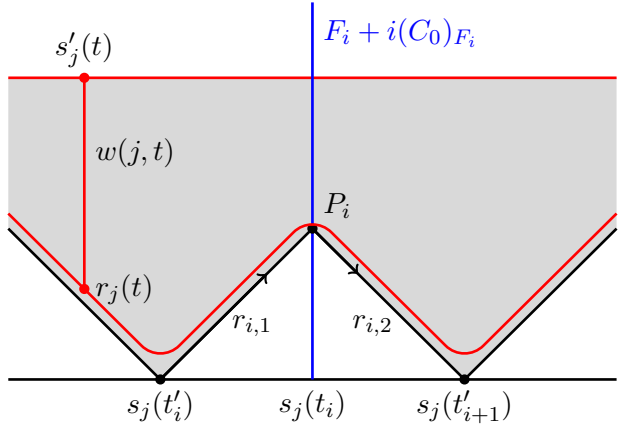

\begin{lemma}\label{lm:tau} In $\fg(\call M(\scr A))$, $p(\omega_i^{(k)} )=
  \tau_i^k$ and $\tau_i\tau_j = \tau_j\tau_i$ for all $i,j$. The
  $\epsilon_*\tau_i$ generate $\fg(T_\Lambda)$.
\end{lemma}
\begin{proof} 
Let $X=f(T_\Lambda)$ be the image of the map $f$ in the proof of 
Lemma \ref{lm:section}, where we now choose $y$ to be a point of our
base chamber $C_0$.

Let the straight line segment $s_j$
be parametrized by
$$ s_j(t ):= t x_0 +
(1-t) u_jx_0,\quad 0\leq t \leq 1.$$

The Minkowski sum $X':= s_1 + \cdots + s_n \subset \mathbb R^n$ is a
fundamental region for the action of $\Lambda$ on $\mathbb R^n$. For
$Y:= X' + iy \subseteq M(\tilde{\scr A})$ we have $p(Y)=X$. In particular, the segments $s_j$ map
under $\epsilon$ to a
system of generators of $\fg(T_\Lambda)$ - in fact, the one associated
with the basis $u_1,\ldots , u_n$ of $\Lambda$.

We will next show that for all $j=1,\ldots , d$ the path
$$
s'_j(t):= s_j(t) + iy
$$
is homotopic to the positive minimal path $\omega_j\in (C_0\to
u_jC_0)$.

Indeed, write $\omega_j=l_{F_1}\ldots l_{F_k}$ and let $t_1,\ldots
t_k$ be such that $s_j(t_i)\in F_i$ for all $i=1,\ldots ,k$. Also,
write $C_i$, $C_{i+1}$ for the source and target chambers of $l_{F_i}$ (note: $C_{k+1}=u_jC_0$) and
for $i=1,\ldots , k-1$ choose $t'_i \in ]t_{i-1}, t_{i}[$, $t'_k:=1$, $t'_0:=0$. Then
$s'_j(t'_i)\in C_i$ for all $i=1,\ldots , k$.

Recall now that the subset of $M(\tilde{\scr A})$ with real part $x\in
F$ consists of points with imaginary part belonging to the chambers of
$\tilde{A}_F$. In fact, the edge $l_{F_i}$, directed from $C_i$ to
$C_{i+1}$, is by construction (\cite[p. 608]{salvetti1987tcr}) the
union of two segments, one from a point in $P'_i\in C_i+0i$ to a point $P_i\in
F+i(C_0)_F$, the other from $P_i$ to a point $P'_{i+1}\in C_{i+1} +
0i$. We will parametrize these segments as $r_{i,1}(t)$, $t'_i\leq t
\leq t_i$ and $r_{i,2}(t)$, $t_i\leq t \leq t'_{i+1}$. Together, they
give a parametrization $r_j(t)$, $0\leq t \leq 1$ of the positive
minimal path $\omega_j$.

The key observation is now that, having chosen $y\in C_0$, we have
that $$s_j(t_h)\in F+ i(C_0)_F \textrm{ for all } h=1,\ldots , k.$$
Since chambers of arrangements are convex, for all $t\in [0,1]$ there
is a straight line segment $w(j,t)$ joining $s_j(t)$ and $r_j(t)$ in
$M(\tilde{\scr A})$.

The (topological) disk $W_j:=\bigcup_{t\in [0,1]}w(j,t)$ defines the
desired homotopy between $s_j$ and $\omega_j$.

Now fix $i,j\in\{1,\ldots ,n\}$ clearly $s_i\, u_i.(s_j)$ is homotopic to $s_j\, u_j.(s_i)$, and in
$\fg(M(\scr A))$ we thus have
$$
\tau_i\tau_j = 
p_*([\omega_i\, u_i.\omega_j])=
p_*([s_i\, u_i.s_j])
$$$$
=p_*([s_j\, u_j.s_i])=
p_*([\omega_j\, u_j.\omega_i])=\tau_j \tau_i.
$$
\end{proof}

\def\tf{\tilde{\call F}}
\begin{defi}
Let $\call Q$ be the set of faces that intersect the
fundamental region $X'$ of the proof of Lemma \ref{lm:tau}. 
Then $\call Q$ contains $C_0$ and $x_0$.
Let $\call Q_i:= \call Q\cap \tf_i$.
In particular, $\call Q_1$ contains the set of faces crossed by $s_i$, for
all $i$.
\end{defi}

\def\otf{\overline{\tilde{\call F }}}
\def\tf{\tilde{\call F}}

Recall the parametrization $s_i(t)$ of the segments $s_i$, and call  $\call B$ the set of
faces of the polyhedron $X'$ which intersect the convex
hull of $\{s_{i}([0,1[)\mid i\in I\}$ for some $I\subseteq \{1,\ldots
,d\}$. Notice that every face of
$X'$ is a translate of some face in $\call B$ by an element $u_1^{m_1}\cdots
u_n^{m_n}$ with $m_1,\ldots ,m_n \in \{0,1\}$.

\begin{defi}\label{df:fbar}
  Let $$\otf := \{F\in \call Q \mid F\cap B=\emptyset \textrm{ for all
  }B\not\in \call B\}$$
\end{defi}

Then $\otf$ is a set of representatives for the orbits of the action
of $\Lambda$ on $\tf$.

\begin{defi} For any given $F\in \tf$ let $\overline F $ be the unique
  element of $\Lambda F\cap\otf$. Then, call $u_F$ the unique element
  of $\Lambda$  such that $F=u_F\overline F$.
  
  Define
  $$
  \Gamma_F:= \omega_{u_F}(u_F.\beta_{\overline{F}})\omega_{u_F}^{-1} 
  $$
\end{defi}
\begin{oss} $\,$\\[-10pt]\begin{itemize}\label{rm:action}
\item[(1)] For all $F\in\tilde{\call F}_1$ and all $u\in \Lambda$
$$p_*(\Gamma_{uF}) = \tau_up_*(\Gamma_F)\tau_u^{-1}.$$

\item[(2)] If $F\in \otf_1$, then $\Gamma_F=\beta_F$.
\item[(3)] If $F\in \call Q$, then $u_F$ has nonnegative coordinates
  with respect to $u_1,\ldots , u_n$. (Recall the discussion before
  Definition \ref{df:fbar}.)
\item[(4)] Since $X'$ is convex, $\call Q_0$ contains the vertices of
  a positive minimal path  between any two elements of $\call Q_0$.
\end{itemize}
\end{oss}

\begin{defi} For $j=1,\ldots ,d$ let
$$
\Omega_j :=\{ F\in \tilde{\call F}_1 : F\textrm{ is crossed by }
\omega_j^{(k)} \textrm{ for some }k\},
$$
And set $\Omega:=\bigcup_j \Omega_j$.
\end{defi}

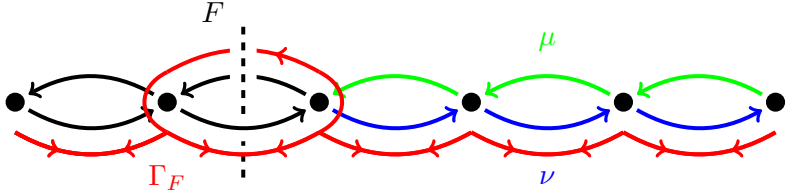
\begin{figure}[tp]
  \centering
  \begin{tikzpicture}[scale = 2]
    \tikzstyle{every path} =
    [black, line width = 1.5pt];
    \tikzstyle{every node} = 
    [circle, fill, inner sep=2pt, outer sep=3pt, draw];

    \foreach \x in {0,...,5}
    \draw node (A\x) at (\x,0) {};

    \foreach \x in {0,...,4} {
      % \pgfmathparse{int(\x+1)} \let\xx\pgfmathresult
      \pgfmathtruncatemacro{\xx}{\x+1}

      \ifthenelse{\x > 1}
      {\path[->] (A\xx) edge [bend right = 30, green] (A\x);}
      {\path[->] (A\xx) edge [bend right = 30] (A\x);};
    }

    \draw[decoration={markings,mark=at position 0.3 with {\arrow{>}}}] 
    (2,.2) edge [red, bend right = 30, postaction={decorate}] (1,.2);

    \draw[white, line width = 10pt] (1.5,-.5) -- (1.5, .5);
    \draw[dashed] (1.5,-.5) -- (1.5, .5);

    \foreach \x in {0,...,4}{
      % \pgfmathparse{int(\x+1)} \let\xx\pgfmathresult
      \pgfmathtruncatemacro{\xx}{\x+1}
      \path (A\xx) edge [white, line width = 10pt, bend left = 30] (A\x);
      
      \ifthenelse{\x > 1} 
      {\path[->] (A\x) edge [bend right = 30, blue] (A\xx);}
      {\path[->] (A\x) edge [bend right = 30] (A\xx);}
      
      \draw[decoration={markings,mark=at position 0.3 with {\arrow{>}}}] 
      (\xx,-.2) edge [red, bend left = 30, postaction={decorate}] (\x,-.2);

      \draw[decoration={markings,mark=at position 0.3 with {\arrow{>}}}] 
      (\x,-.2) edge [red, bend right = 30, postaction={decorate}] (\xx,-.2);
    }

    \draw[decoration={markings,mark=at position 0.3 with {\arrow{>}}}] 
    (0,-.2) edge [red, bend right = 30, postaction={decorate}] (1,-.2);
    
    \path [draw, red] (1,-.2) .. controls (.8,-.1) and (.8,.1) .. (1,.2)
    (2,-.2) .. controls (2.2,-.1) and (2.2,.1) .. (2,.2);

    \path node [draw = none, fill = none, blue] (N) at (3.5,-.5) {$\nu$}
    node [draw = none, fill = none, green] (M) at (3.5,.4) {$\mu$}
    node [draw = none, fill = none, red] (G) at (1,-.5) {$\Gamma_F$}
    node [draw = none, fill = none] (F) at (1.3,.6) {$F$};
  \end{tikzpicture}
  \caption{Construction for the proof of lemma \ref{lemma:step1}}
  \label{fig:step1}
\end{figure}

\begin{lemma}\label{lemma:step1} For all $i=1,\ldots, n$, the subgroup of $\fg(M(\tilde{\scr
    A}))$ generated by the elements $\beta_F$ with $F\in \Omega_i$ is
  contained in the subgroup generated by the $\Gamma_F$, $F\in \Omega_i$.
\end{lemma}
\begin{proof}
Let w.l.o.g. $F\in \Omega_1$, and say that $F=u_1^k\overline F$. 
If $k\geq 0$, by construction we have $\Gamma_F=\beta_F$.

Suppose then $k<0$, and in this case $C':=(C_0)_F\neq
(u_1^kC_0)_F$.  Let $\nu$ denote the positive minimal path from
$C'$ to $C_0$ that follows the segments $s_1$. We argue by induction
on the length $d(F)$ of $\nu$: if $d(F)= 0$ we have in fact
$\Gamma_F=\beta_F$.

Now let $d(F)>0$. Then 
$$
\Gamma_F\simeq \nu^{-1} l_F^2 \nu;\quad \beta_F=\mu l_F^2 \mu^{-1}
$$
where $\mu$ is the positive minimal path from $C_0$ to $C'$ following
 $s_1$. Thus 
$$
\beta_F= \mu\nu\nu^{-1} l_F^2 \nu (\mu\nu)^{-1} = (\mu\nu) \Gamma_F (\mu\nu)^{-1}
$$
where $\mu\nu$ is the product of all $\beta_{F'}$ with $F'$ crossed by
$\mu$ - therefore, with $F'\in \Omega_1$ and $d(F')<d(F)$. By induction, the claim follows.
\end{proof}

\begin{lemma}\label{lm:fundgen} The set $\{\Gamma_F \mid F\in \Omega\}$
  generates $\fg(\call M (\tilde{\scr A}))$.
\end{lemma}
\begin{proof}
Let $F\in \tf_1$, and let $H$ the affine hyperplane supporting $F$.

By construction, there is $i\in \{1,\ldots ,d\}$ and $k\in \mathbb Z$
such that $H$ is crossed by $\omega_i^{(k)}$ in, say, the face $G$
(`every hyperplane is cut by the coordinate axes').

By Lemma \ref{lemma:12}, $\beta_F$ is then product of $\beta_G$ and
other $\beta^\pm_{G'}$ with $G'\in\Omega$. These can be written in
terms of the $\Gamma_F$ by Lemma \ref{lemma:step1}.
\end{proof}

\subsection{Relations}
We now turn to the study of the relations.

\begin{lemma}\label{lm:Q1}
Let $F\in\call Q_1$. Then there is a sequence $F_1,\ldots,F_k$ of
elements of $\call Q_1$
such that $\beta_F$ is homotopic to
$$
(\prod_{i=1}^k\Gamma_{F_i})^{-1} \Gamma_{{F}}(\prod_{i=1}^k\Gamma_{F_i}).
$$

Moreover, $(F_1,\ldots,F_k)=\Sigma({\omega_{u_F}\, (u_FC_0\to
  (u_FC_0)_F)})$ as in Equation \ref{eq:sigma}. In particular, the $F_i$
are translates of elements of $\Omega\cap \otf$. 
\end{lemma}

\begin{proof} 
By definition $\Gamma_F=\omega_{u_F}u_F.\beta_{\overline F}
\omega_{u_F}^{-1}$. Writing $\mu$ for a positive minimal path
$(u_FC_0\to (u_FC_0)_F)$ we decompose this into 
$$
\Gamma_F=\omega_{u_F} \mu (l_F)^2 (\omega_{u_F}\mu)^{-1}.
$$

With Remark \ref{rm:action}.(3) we have that $\omega_{u_F}\mu$ is a
positive path, and with Lemma \ref{lm:salv12} we write it as a
product $\prod_j \beta^{\omega_{u_F}\mu}_{G_j} (C_0\to (C_0)_F)$ where since $\mu$ is
positive miminal, the $G_j$ are crossed by $\omega_{u_F}$ and thus are translates of faces intersecting the
segments $s_i$. 

Now, by construction
$$ \beta^{\omega_{u_F}\mu}_{G_j} = \Gamma_{G_j}.$$
Then, set
$$
\Delta_F:=\prod_j \Gamma_{G_j}. 
$$

Therefore if $(C_0)_F=(u_FC_0)_F$ we are done with
$$ \Gamma_F \simeq\Delta_F \beta_F \Delta_F^{-1}, \textrm{ and thus }
\beta_F \simeq \Delta_F^{-1} \Gamma_F \Delta_F.$$

If $(C_0)_F\neq (u_FC_0)_F$, then we may choose a representant of
$(C_0\to (u_FC_0)_F)$ that ends with $l_F$, so its inverse begins
with $l_F^{-1}$ and we have the same relation as above.
\end{proof}

Keeping the notations of the Lemma we define, for every $F\in\call
Q_1$,
\begin{equation}\label{eq:delta}
\Delta_F:=\prod_{G\in \Sigma(\omega_{u_F}(u_FC_0\to (u_FC_0)_F))}\Gamma_{G};\quad\quad 
\Gamma^\Delta_F:={\Delta_F^{-1}}\Gamma_F^{\,}\Delta_F^{\,}
\end{equation}

Recall from \ref{subsec:gen}.II that to every face
$G\in\tf_2$ we have an ordered set $h(G)=(F_1,\ldots, F_k)$ of
incident codimension 1 faces, one for every hyperplane containing
$G$. The relations associated with $G$ assert the equality of
\begin{equation}\label{eq:rel1}
\beta_{F_{\sigma(1)}}\ldots \beta_{F_{\sigma(k)}}
\end{equation}
where $\sigma$ is a cyclic permutation, and we write $\beta_i$ for
$\beta_{F_i}$.

\begin{lemma}\label{lm:conjrel} Given $G\in \tf_2$ there is
  $\Delta_G$ such that, for all cyclic permutations $\sigma$, we have
  a homotopy of paths 
$$
\beta_{F_\sigma(1)}\ldots \beta_{F_\sigma(k)} \simeq 
\Delta_G\omega_{u_G}
u_G.(\Gamma^\Delta_{u_G^{-1}F_{\sigma(1)}}\ldots \Gamma^{\Delta}_{u_G^{-1}F_{\sigma(k)}})
\omega_{u_G}^{-1}\Delta_G^{-1}.
$$
\end{lemma}
\def\oug{\omega_{u_G}}
\begin{proof} Let us fix some notation and let
$C':=(C_0)_G$, $C'':=(u_G.C_0)_G $, $\mu:=(u_GC_0\to C'')$,
$\nu:=(C''\to C')$.
By equation (\ref{eq:alfabeta}) we have the homotopy
$$
 \beta_{\sigma(1)}\ldots \beta_{\sigma(k)} \simeq  (C_0\to C')\alpha_G(C')(C_0\to C')^{-1} 
$$
moreover, with Equation (\ref{eq:alfa}) we see
$$
\alpha_G(C') \simeq \nu^{-1}\alpha_G(C'') \nu
\simeq
\nu^{-1}
\mu^{-1} \oug^{-1}\oug \mu
\alpha_G(C'') 
\mu^{-1}\oug^{-1}\oug \mu \nu
$$

expanding $\mu
\alpha_G(C'') 
\mu^{-1}$ according to Equation (\ref{eq:alfabeta}) and defining
$\Delta_G:= (C_0\to C')\nu^{-1}\mu^{-1}\oug^{-1}$ we have the homotopy

\begin{equation}\label{eq:expanded}
 \beta_{\sigma(1)}\ldots \beta_{\sigma(k)}  \simeq
\Delta_G\oug
(u_G.\beta_{u_G^{-1}F_{\sigma(1)}})\ldots (u_G.\beta_{u_G^{-1}F_{\sigma(k)}})
\oug^{-1}\Delta_G^{-1}
\end{equation}

From which the claim follows by use of Lemma \ref{lm:Q1}.
\end{proof}

\def\rel{R^{\downharpoonright}}
\begin{defi}\label{df:genreltorico} For $F\in \tf_1$ let 
$$
\gamma_F:=p(\Gamma_F).
$$
Moreover, for $F\in \call Q_1$ let
$$ \delta_F:=p(\Delta_F); \quad\quad \gamma_F^\delta:=\delta_F^{-1}\gamma_F\delta_F
$$

Given $G\in \tf_2$ with $h(G)=(F_1,\ldots,F_k)$, 
let $\rel_G$ define the relation stating the
equality of all words
$$
\gamma_{F_{\sigma(1)}}^\delta \cdots \gamma_{F_{\sigma(k)}}^\delta
$$
where $\sigma$ ranges over all cyclic permutations.
\end{defi}

\begin{lemma}\label{lm:relconj}
If $G\in \tf_2$ is a face of codimension $2$, then $\rel_G$ is equivalent to
$\rel_{\overline G}$
\end{lemma}
\begin{proof}
Let $G\in \tf_2$. With Lemma \ref{lm:conjrel} (and the notation thereof) we know that every
relation $\rel_G$ 
states the equality of all
$$
p_*(\Delta_G) p_*(\Gamma^\Delta_{F_{\sigma(1)}}\ldots \Gamma^{\Delta}_{F_{\sigma(k)}}) p_*(\Delta_G)^{-1},
$$
where $\sigma$ runs over all cyclic
permutations. The middle term by Equation
(\ref{eq:expanded}) is represented by the path
$$
\oug
(u_G.\beta_{u_G^{-1}F_{\sigma(1)}})\ldots (u_G.\beta_{u_G^{-1}F_{\sigma(k)}})
\oug^{-1}
$$

and thus its image under $p_*$ is represented by the same path as
$$
p_*(\omega_{u_G}) 
p_*(\beta_{u_G^{-1}F_{\sigma(1)}}\ldots \beta_{u_G^{-1}F_{\sigma(k)}})
p_*(\omega_{u_G})^{-1}
$$

Where $u_G^{-1}F_{\sigma(i)}\in\call Q_1$ for all $i$.
Now we apply Lemma \ref{lm:Q1}. The element
$\mu:=p_*(\omega_{u_G})\in\fg(T_\Lambda)$ is such that, for every
cyclic permutation $\sigma$,
$$
p_*(\Gamma^\Delta_{F_{\sigma(1)}}\ldots
\Gamma^{\Delta}_{F_{\sigma(k)}})
=
\mu \,
p_*(\Gamma^\Delta_{\overline F_{\sigma(1)}}\ldots
\Gamma^{\Delta}_{\overline F_{\sigma(k)}})
\mu^{-1}
$$
and therefore relation $\rel_G$ is equivalent to relation $\rel_{\overline
  G}$.
\end{proof}

\subsection{Presentation}

In this closing section we discuss presentations for $\fg(M(\scr A))$.

\begin{lemma}
For all $F\in \call Q_1$ let $(F_1,\ldots
F_k)=\Sigma(\omega_{u_F}(u_FC_0\to (u_FC_0)_F))$. We have 
$$
\delta_F = \prod_{i=1}^k \tau_{u_{F_i}}\gamma_{\overline F_i} \tau_{u_{F_i}}^{-1}
$$
and, in particular, $\gamma_F^\delta$ can be written as a word in the
$\tau_1,\ldots ,\tau_n$ and $\gamma_F$ with $F\in \otf_1$.
\end{lemma}
\begin{proof} This is an easy computation using Remark \ref{rm:action}.(1).
\end{proof}

In Particular, the relations $\rel$ can be written in terms of the
$\tau_i$ and the $\gamma_F$ with $F\in \otf_1$. We have immediately

\begin{thmA}\label{teo:mainfg}
The group $\fg(\call M (\scr A))$ is presented as
$$
\langle \tau_1,\ldots, \tau_n; \; \gamma_F,\, F\in \call F_1 \mid
\tau_i\tau_j=\tau_j\tau_i \textrm{ for }i,j=1,\ldots , n;\; \rel_G,\, G\in \call F_2\rangle,
$$
where we identify $\call F_1$ with $\otf$ and $\call F_2$ with $\otf_2$.
\end{thmA}

This presentation, while not very economical in terms of generators,
has the advantage that the relations can be described with an
acceptable amount of complexity. 

Using Lemma \ref{lm:fundgen} and Remark \ref{rm:action}.(1) we can
let, for all $G\in \otf_2$, $\widetilde{R}_G^\downharpoonright$ denote the relations
obtained from $\rel_G$ by substituting every $\gamma_F$ with the
corresponding expression in terms of the generators $\tau_1,\ldots ,
\tau_d$ and $\gamma_{F'}$ with $F'\in\otf\cap \Omega$. Under the
identification of $\call F_1$ with $\otf$, these are the faces on the
compact torus that are crossed
by some fixed chosen reppresentants of the generators $\tau_1,\ldots ,\tau_d$.

\begin{thmA}
The group $\fg(\call M (\scr A))$ is presented as
$$
\langle \tau_1,\ldots, \tau_n; \; \gamma_F,\, F\in p(\Omega) \cap
\call F_1 \mid
\tau_i\tau_j=\tau_j\tau_i \textrm{ for }i,j=1,\ldots , n;\; \widetilde{R}^\downharpoonright_G,\, G\in \call F_2\rangle.
$$
\end{thmA}

\begin{oss}
The number of generators (and relations) can in principle be reduced
further, by adequate choice of the coordinates of $T_\Lambda$. The
computations, however, become quite more involved and
untransparent. We thus omit them here, leaving the question open for a presentation with generators and relations
corresponding to layers instead of faces (which exists in the case of
complexified hyperplane arrangements, as shown by Salvetti in
\cite{salvetti1987tcr} by simplifying the presentation given above in \ref{subsec:gen}.3).
\end{oss}

\ifthenelse{\boolean{appendice}}{
\appendix
\section{Covering maps and quotients}
We will use here the notations of \cite{bridson1999metric}.

\begin{prop}\label{prop:uniquelift}
  Let $p:\call C \to \call D$ be a covering of acyclic categories,
  $\gamma = (e_1, \dots, e_k)$ an edge path in $\call D$ joining
  $d$ and $d'$, $c \in p^{-1}(d)$ a base point.
  Then there is a unique \emph{lift} $\delta = (f_1, \dots, f_k)$,
  such that $p(\delta) = \gamma$ and $i(f_1) = c$.
\end{prop}

\begin{defi}
  Let $p:\call C \to \call D$ be a covering of acyclic categories,
  the \emph{automorphism group} of $p$ is the group of 
  automorphisms $\varphi$ of $\call C$, such that $p\circ\varphi = p$.
\end{defi}

The automorphism group of a covering acts of course on the covering space
$\call C$. The following fact follows from proposition \ref{prop:uniquelift}
\begin{prop}
  Let $p:\call C \to \call D$ be a covering of acyclic categories,
  then the action of $\mbox{Aut}(p)$ on $\call C$ is free.
\end{prop}

\begin{defi}
  A covering of acyclic categories $p:\call C \to \call D$ is said to be
  a \emph{Galois covering} if the action of $\mbox{Aut}(p)$ 
  is transitive on every fibre $p^{-1}(d)$. In this case $\mbox{Aut}(p)$ 
  is also called the \emph{Galois group} of $p$.
\end{defi}

\begin{prop}\label{prop:rivestimentoquoziente}
  Let $p: \call C \to \call D$ be a Galois covering of acyclic categories,
  $G = \mbox{Aut}(p)$. Then there is an isomorphism $\Psi$ of covering maps,
  such that the following diagram commutes:
  \begin{displaymath}
    \xymatrix{
      & \call D \ar@/^/[dd]^\Psi\\
      \call C \ar[ur]^p \ar[dr]_\pi & & \\
      & \call C/G
    }
  \end{displaymath}
\end{prop}
\begin{proof}
  We already know that $G$ acts transitively on the fibers
  $p^{-1}(d)$ with $d \in \mbox{Ob}\,\call D$. $G$ must act transitively also
  on the fiber $p^{-1}(m)$ of the morphisms
  $m \in \mbox{Mor}\,\call D$. Indeed, let $g m = m$ for some $g \in G$
  and $m \in \mbox{Mor}\,\call D$, then $gi(m) = i(m)$.
  Therefore the map $\Psi(p(m)) = \pi(m)$ 
  is well-defined, bijective and is the unique set map that let the diagram
  commute.

  We still need to prove that $\Psi$ is a functor. That is that,
  $\Psi(p(m)\circ p(n)) = \Psi(p(m)) \circ \Psi(p(n))$. We can choose
  $m$ and $n$ to be composable, so that
  \begin{displaymath}
    \Psi(p(m)\circ p(n)) = 
    \Psi(p(m\circ n)) = \pi(m\circ n) =
    \pi(m) \circ \pi(n) = \Psi(p(m)) \circ \Psi(p(n)).
  \end{displaymath}
  
  To prove that $\Psi$ is an isomorphism, we need still to show that
  $\Psi^{-1}$ (defined as a set map) is a functor. %\marginpar{\it c'\`e bisogno?}.
  Since $\call C/G$ is a quotient both in the category of sets and in those
  of acyclic categories, there is a unique map $\Phi: \call C/G \to \call D$
  such that $\Phi\circ\phi = p$, and this map is a functor. It then
  needs to be $\Phi = \Psi^{-1}$.
\end{proof}
%IFTHEN CLOSE
}{}

\bibliographystyle{plain}
\bibliography{toric}

\end{document}